\documentclass[11pt]{article}

\usepackage{natbib}
\usepackage{fullpage, parskip}
\usepackage{graphicx}
\usepackage{multirow}
\usepackage{amsmath,amssymb,amsfonts}
\usepackage{amsthm}
\usepackage{mathrsfs}
\usepackage{mathtools}
\usepackage[title]{appendix}
\usepackage{xcolor}

\usepackage{hyperref}
\definecolor{mydarkblue}{rgb}{0,0.08,0.45}
\definecolor{mydarkred}{rgb}{0.75, 0.27, 0.34}
\hypersetup{
    colorlinks=true,
    citecolor=mydarkblue,
    linkcolor=mydarkred,
}

\usepackage{manyfoot}
\usepackage{booktabs}
\usepackage{etoolbox}
\usepackage{url}
\usepackage{xspace}
\usepackage{subcaption, multicol, multirow, vwcol}
\usepackage[ruled,lined,linesnumbered,noend]{algorithm2e}
\usepackage{thm-restate}
\usepackage{pgfplots,tikz}
\pgfplotsset{compat=1.15}

\usepackage[capitalise,noabbrev]{cleveref}

\crefname{section}{Section}{Sections}
\crefname{subsection}{Section}{Sections}
\crefname{subsubsection}{Section}{Sections}
\crefname{algocf}{Algorithm}{Algorithms}
\crefname{equation}{Equation}{Equations}
\crefname{assumption}{Assumption}{Assumptions}
\crefname{observation}{Observation}{Observations}
\crefname{claim}{Claim}{Claims}

\newtheorem{theorem}{Theorem}
\newtheorem{lemma}{Lemma}
\newtheorem{corollary}{Corollary}

\newtheorem{proposition}{Proposition}

\theoremstyle{definition}
\newtheorem{definition}{Definition}
\newtheorem{assumption}{Assumption}

\crefformat{theorem}{\textcolor{black}{Theorem}~#2#1#3}
\Crefformat{theorem}{\textcolor{black}{Theorem}~#2#1#3}
\crefformat{lemma}{\textcolor{black}{Lemma}~#2#1#3}
\Crefformat{lemma}{\textcolor{black}{Lemma}~#2#1#3}
\crefformat{corollary}{\textcolor{black}{Corollary}~#2#1#3}
\Crefformat{corollary}{\textcolor{black}{Corollary}~#2#1#3}
\crefformat{proposition}{\textcolor{black}{Proposition}~#2#1#3}
\Crefformat{proposition}{\textcolor{black}{Proposition}~#2#1#3}
\crefformat{definition}{\textcolor{black}{Definition}~#2#1#3}
\Crefformat{definition}{\textcolor{black}{Definition}~#2#1#3}
\crefformat{assumption}{\textcolor{black}{Assumption}~#2#1#3}
\Crefformat{assumption}{\textcolor{black}{Assumption}~#2#1#3}
\crefformat{observation}{\textcolor{black}{Observation}~#2#1#3}
\Crefformat{observation}{\textcolor{black}{Observation}~#2#1#3}
\crefformat{claim}{\textcolor{black}{Claim}~#2#1#3}
\Crefformat{claim}{\textcolor{black}{Claim}~#2#1#3}
\crefformat{remark}{\textcolor{black}{Remark}~#2#1#3}
\Crefformat{remark}{\textcolor{black}{Remark}~#2#1#3}
\crefformat{example}{\textcolor{black}{Example}~#2#1#3}
\Crefformat{example}{\textcolor{black}{Example}~#2#1#3}
\crefformat{section}{\textcolor{black}{Section}~#2#1#3}
\Crefformat{section}{\textcolor{black}{Section}~#2#1#3}
\crefformat{subsection}{\textcolor{black}{Section}~#2#1#3}
\Crefformat{subsection}{\textcolor{black}{Section}~#2#1#3}
\crefformat{subsubsection}{\textcolor{black}{Section}~#2#1#3}
\Crefformat{subsubsection}{\textcolor{black}{Section}~#2#1#3}
\crefformat{appendix}{\textcolor{black}{Appendix}~#2#1#3}
\Crefformat{appendix}{\textcolor{black}{Appendix}~#2#1#3}
\crefformat{figure}{\textcolor{black}{Figure}~#2#1#3}
\Crefformat{figure}{\textcolor{black}{Figure}~#2#1#3}
\crefformat{table}{\textcolor{black}{Table}~#2#1#3}
\Crefformat{table}{\textcolor{black}{Table}~#2#1#3}
\crefformat{algocf}{\textcolor{black}{Algorithm}~#2#1#3}
\Crefformat{algocf}{\textcolor{black}{Algorithm}~#2#1#3}
\crefformat{equation}{\textcolor{black}{Equation}~#2#1#3}
\Crefformat{equation}{\textcolor{black}{Equation}~#2#1#3}

\allowdisplaybreaks

\newcommand{\poly}{\mathrm{poly}}

\newcommand{\ignore}[1]{}

\newcommand{\E}{\mathbb{E}}

\newcommand{\conv}{\mathrm{conv}}

\newcommand{\simplex}{\mathbf{\Delta}}

\newcommand{\R}{\mathbb{R}}

\newcommand{\convexbody}{\mathcal{K}}
\newcommand{\regret}{\textup{\texttt{regret}}}
\newcommand{\ent}{\mathrm{ent}}
\newcommand{\euc}{\mathrm{euc}}
\newcommand{\var}[2]{#1^{(#2)}}
\newcommand{\gain}{\textup{\texttt{gain}}}

\usepackage{glossaries}

\newacronym{oco}{OCO}{Online Convex Optimization}
\newacronym{omd}{OMD}{Online Mirror Descent}
\newacronym{opgd}{OPGD}{Online Projected Gradient Descent}
\newacronym{oeg}{OEG}{Online Exponentiated Gradient}

\title{Block-Norm Geometries for Online Mirror Descent \\ with Sparse Losses}

\author{
    Swati Gupta \\ MIT \and
    Jai Moondra\footnote{Corresponding author, email: \texttt{jmoondra@andrew.cmu.edu}}  \footnote{Part of this work was done while the author was at Georgia Tech} \\ CMU \and
    Mohit Singh \\ Georgia Tech
}
\date{}

\begin{document}

\maketitle

\begin{abstract}
    The performance of online mirror descent depends critically on the geometry induced by its mirror map, yet standard algorithms largely rely on two canonical choices: Euclidean and entropic geometry.
    We show that these two geometries can both be substantially suboptimal when loss gradients are sparse.
    We introduce a family of randomized block-norm mirror maps that interpolates between Euclidean and entropic geometries and adapts to intermediate sparsity structure.
    For several standard convex sets, including \(\ell_p\) balls, ellipsoids, boxes, and Minkowski sums of norm balls, we prove polynomial-in-dimension improvements in regret bounds over the better of online projected gradient descent and exponentiated gradient.
    We further construct explicit online convex optimization instances for which these improvements are realized: on a simple polytope, an intermediate block geometry achieves a $\poly(d)$ separation in regret from both Euclidean and entropic geometries in dimension $d$, while on the probability simplex we obtain a separation of order \(\Omega(\sqrt{\log d}/\log\log d)\).
    Finally, we study geometry selection when sparsity is unknown.
    We show that naïvely alternating between mirror maps can incur linear regret, even though either mirror map alone has sublinear regret, and give a Hedge meta-algorithm that competes with the best mirror map in a finite portfolio.
    For random block geometries, this yields regret within an \(O(\sqrt{\log\log d})\) factor of the best random uniform block norm chosen in hindsight.

\end{abstract}

\section{Introduction}

Choosing the right geometry is a central design question in optimization.
For example, the same optimization problem can behave very differently depending on the norm, distance-generating function, or coordinate system used by an algorithm.
In mirror descent algorithms, choosing a mirror map that is well matched to the geometry of the feasible region may make projections and iterations efficient, while one that is well matched to the objective or gradient structure may yield substantially sharper convergence or regret guarantees.
Classical examples already suggest that there is no universally best choice: Euclidean geometry is natural for some domains, while entropic geometry is far better suited to others
\cite{kivinen1997exponentiated, littlestone1994weighted, cesa1997use, hazan_introduction_2016, ben-tal_lectures_2001, srebro2011universality, gatmiry2025computing, cheng2019geometry}.
More broadly, this raises the question of whether geometry should be viewed not as a fixed modeling choice, but as an algorithmic degree of freedom that can be adapted to the structure of the problem.
In this paper, we study this question in \gls{oco}, where the appropriate geometry must simultaneously reflect the shape of the feasible set and the structure of the losses revealed over time.

\gls{oco} provides a broadly applicable setting in which to study this question.
At each time step \(t \in [T]\), a learner chooses a feasible point $\var{x}{t}$ in a convex feasible set $\convexbody \subseteq \R^d$, after which an adversary reveals a convex loss function $\var{f}{t}$.
The learner then uses the revealed loss function $\var{f}{t}$ to choose its next decision, with the goal of minimizing \emph{regret} relative to the best fixed feasible decision $x^* = {\arg\min}_{x \in \convexbody} \sum_{t \in [T]} \var{f}{t}(x)$ in hindsight:
\begin{equation}\label{eqn: regret-oco}
    \regret(T) = \sum_{t \in [T]} f^{(t)}(x^{(t)}) - \min_{x \in \mathcal{K}} \sum_{t \in [T]} f^{(t)}(x).
\end{equation}

The generality of this paradigm allows modeling various practical problems, including recommendation systems, online supervised learning, clinical trials, financial portfolios, adaptive routing, and stochastic optimization (see \cite{hazan_introduction_2016}).

\gls{omd} \cite{nemirovski1983problem, zinkevich_online_2003, shalev-shwartz_online_2012, hazan_introduction_2016} is a class of projection-based algorithms for \gls{oco} that generalizes \gls{opgd} by transforming the geometry of $\convexbody$ using a \emph{mirror map} $h: \convexbody \to \R$, which is a strictly convex and differentiable function.
\gls{omd} reduces to \gls{opgd} when $h_\euc(x) := \tfrac{1}{2} \|x\|_2^2$ is half the squared $L_2$ norm, corresponding to \emph{Euclidean geometry} \cite{zinkevich_online_2003}, and to \gls{oeg}, which uses $h_\ent(x) := \sum_{i} x_i \ln x_i$ that corresponds to the $L_1$ norm or the \emph{entropic} geometry \cite{littlestone1994weighted, cesa1997use}.

Zinkevich \cite{zinkevich_online_2003} proved that \gls{opgd} over time horizon $T$ has regret at most $O(D G \sqrt{T})$, where $D$ is the Euclidean diameter of feasible set $\convexbody$, and $G \ge \max_{t \in [T]} \sup_{x \in \convexbody} \|\nabla f^{(t)}(x)\|_2$ is an upper bound on the Euclidean norm of loss gradients.
The corresponding guarantee for \gls{omd} \cite{shalev-shwartz_online_2012, hazan_introduction_2016} depends on the geometry induced by a chosen mirror map: formally, the \emph{Bregman divergence} $B_h$ associated with the mirror map $h$ is defined as
\begin{equation}\label{eqn: bregman-divergence}
    B_h(x \| y) = h(x) - h(y) - \nabla h(y)^\top(x-y),
\end{equation}
which is the additive gap between $h(x)$ and the first-order approximation of $h(x)$ at $y$.
The OMD updates proceed as follows:
\begin{equation}\label{eqn: omd-update}
    x^{(t+1)} = {\arg \min}_{x \in \mathcal{K}} \big(\nabla f^{(t)}(x^{(t)})\big)^\top x + \tfrac{1}{\eta} B_{h}(x \|  x^{(t)}).
\end{equation}
For \emph{standard} step size $\eta := \tfrac{D_h}{G_h \sqrt{T}}$, the regret of \gls{omd} is upper bounded by $2 D_h G_h \sqrt{T}$, where $D_h := \max_{x \in \mathcal{K}}\sqrt{B_{h}(x \| x^{(1)})}$ is defined as the \emph{diameter} of $B_h$ for given starting point $x^{(1)} \in \convexbody$, the mirror map $h(\cdot)$ is 1-strongly convex with respect to the norm $\|\cdot\|$, and $G_h := \max_{t} \sup_{x \in \convexbody} \|\nabla f^{(t)}(x)\|^*$ is the Lipschitz constant of the loss functions under the dual norm $\|\cdot\|^*$.

This raises the natural question: \emph{given a convex feasible set $\convexbody$ and structural constraints on the loss functions, how do we characterize the optimal mirror map $h$?}

\paragraph{Sparse losses.}
Most previous work has focused on either (1) specific feasible sets such as the probability simplex, Euclidean ball, or $L_p$ norm balls etc \cite{cesa1997use, kwon2016gains, srebro2011universality}, or (2) convex loss gradients, i.e., loss functions whose gradients lie in some convex set $\mathcal{G} \subseteq \R^d$.
In this work, we study loss functions with sparse gradients, where at most $s$ coordinates of the gradient vector of the losses are non-zero at any feasible point, for some $s \ll d$.
The set $\mathcal{G}_s$ of $s$-sparse vectors is \emph{non-convex} and has only been studied for special convex bodies \cite{kwon2016gains}.
It arises naturally in many practical learning setups, motivating a systematic study:

\begin{itemize}
    \item In \emph{supervised Machine Learning}, most datasets consist of numerical as well as categorical (non-numerical) features.
    Most training algorithms (such as stochastic gradient descent) first convert each categorical feature into a binary feature vector using one-hot encoding.
    The length of this binary vector equals the number of categories, indicating the category for the given data point.
    The cumulative feature vector is therefore sparse for each data point, since all but one of the entries for each categorical feature are zero.

    For models that are based on linear/polynomial regression, logistic regression, and support vector methods etc, the corresponding loss function is of the form $\var{f}{t}(w) = \ell(w^\top \var{z}{t}, \var{y}{t})$ where $\var{z}{t} \in \R^d$ is a data point, $\var{y}{t} \in \R$ is the corresponding label, and $w \in \R^d$ is the parameter set to be learnt.
    For such models, sparsity of feature vectors $\var{z}{t}$ leads to sparse loss gradients, since $\tfrac{\partial \var{f}{t}}{\partial w_i} = \var{z}{t}_i \tfrac{\partial \ell}{\partial a}$, where $a = w^\top \var{z}{t}$.
    Therefore, $\tfrac{\partial \var{f}{t}}{\partial w_i} = 0$ when $\var{z}{t}_i = 0$.

    \item In \emph{combinatorial semi-bandit learning} setups, e.g., learning over shortest paths in a dense congestion network, or online matchings in a dense bipartite graph, the observed costs give rise to sparse gradient estimates \cite{audibert2014regret}.
    For example, in online-shortest paths, a decision-maker repeatedly picks a path from a source to a destination in a known graph, with edge costs (congestion) changing over time, and observes costs only on the chosen path.
    In dense networks, the inherent dimension (total number of edges) is much larger than the number of edges in any single path, so each resulting gradient estimate is supported on only a small subset of coordinates.

    \item  Sparse loss gradients can also arise in \emph{distributed learning} algorithms, where the loss gradient is distributed into several sparse sub-parts for parallelization \cite{wangni2018gradient}.
\end{itemize}

Formally, we analyze the following class of loss functions for \emph{oblivious} adversaries:
\begin{assumption}[Sparse gradients]\label{assmp: loss-functions}
    We assume that every loss function $f$ has bounded gradients supported in coordinates $C_f\subseteq[d]$ of size at most $s$, that is, for all $f$
    \begin{enumerate}
        \item for all $x \in \convexbody$, $\nabla f(x) \in [-1, 1]^d$, and
        \item there is some coordinate set $C_f \subseteq [d]$ with $|C_f| \le s$ such that $\nabla_i f(x) = 0$ for every $x \in \convexbody$ and every $i \not\in C_f$.
    \end{enumerate}
\end{assumption}

Note that as the loss functions $\var{f}{t}$ vary over time, the support coordinate set $\var{C_f}{t}$ also changes over time.
Additionally, the bounds $[-1, 1]^d$ may be replaced by $[-\sigma, \sigma]^d$ for any $\sigma > 0$ without loss of generality.

\begin{assumption}[Oblivious adversary]
    We assume that the loss functions are generated by an oblivious adversary so that the (possibly random) sequence of loss functions $\var{f}{1}, \dots, \var{f}{T}$ is fixed at $t = 0$ and does not depend on the player's moves.
\end{assumption}

\paragraph{Key questions.}
Determining the optimal mirror map for given convex body $\mathcal{K} \subseteq \mathbb{R}^d$ and sparse losses remains a challenging open question.
The core difficulty lies in optimizing the fundamental trade-off governing the regret of the OMD: $D_{h}(\mathcal{K}) G_h(\mathcal{G})$ for given $(\mathcal{K}, \mathcal{G})$, and jointly reasoning about these.

A canonical example is learning over the probability simplex: under the entropic geometry, the simplex has diameter $D_\ent = \Theta(\sqrt{\ln d})$, and $G_{\ent} = \max_{x \in \convexbody, f: \nabla f \in \mathcal{G},} \|\nabla f(x)\|_\infty$ under the dual $L_{\infty}$ norm.
On the other hand, under Euclidean geometry, the diameter decreases to $D_\euc = \Theta(1)$, but the Lipschitz constant of the gradients of the losses increases to $G_\euc = \max_{x \in \convexbody, f: \nabla f \in \mathcal{G}, } \|\nabla f(x)\|_2$, which can be a factor $\sqrt{s}$ larger than $G_\ent$ for $s$-sparse losses.
Therefore, for $s \gg \log d$, \gls{oeg} incurs lower regret than \gls{opgd}.
Kwon and Perchet \cite{kwon2016gains} showed that using mirror maps based on $L_p$ norms can improve the regret guarantee further if the parameter $p \in [1, 2]$ is tuned to sparsity $s$.

More generally, we would like to find a broad family or \emph{portfolio} $M$ of mirror maps, so that the performance under the optimal mirror map for any $(\mathcal{K}, \mathcal{G})$ is well-approximated by some $h \in M$.
This work focuses on a simpler question:
\begin{center}
    \it (a) Is there a family of mirror maps $M$ that can adapt to the sparsity of the loss functions better than the classical alternatives \gls{opgd} and \gls{oeg}?
    \\ (b) Can we find the optimal mirror map in this family online when sparsity $s$ is unknown?
\end{center}

The natural candidate family $M$ of convex combinations of the entropic mirror map $h_\ent(x) = \sum_i x_i \ln x_i$ and Euclidean mirror map $h_\euc(x) = (1/2) \|x\|_2^2$ is not useful: as we show in \cref{app: convex-combination-is-not-useful}, for losses with $s$-sparse gradients over the probability simplex $\simplex_d$, the regret upper bound \cref{eqn: regret-bound-general} for any $\lambda$ always exceeds the better of $h_\euc$ and $h_\ent$.
Instead, we consider a different family interpolating between these geometries:

\begin{figure}[!t]
    \centering
    \begin{minipage}[c]{0.58\linewidth}
        \includegraphics[width=\linewidth]{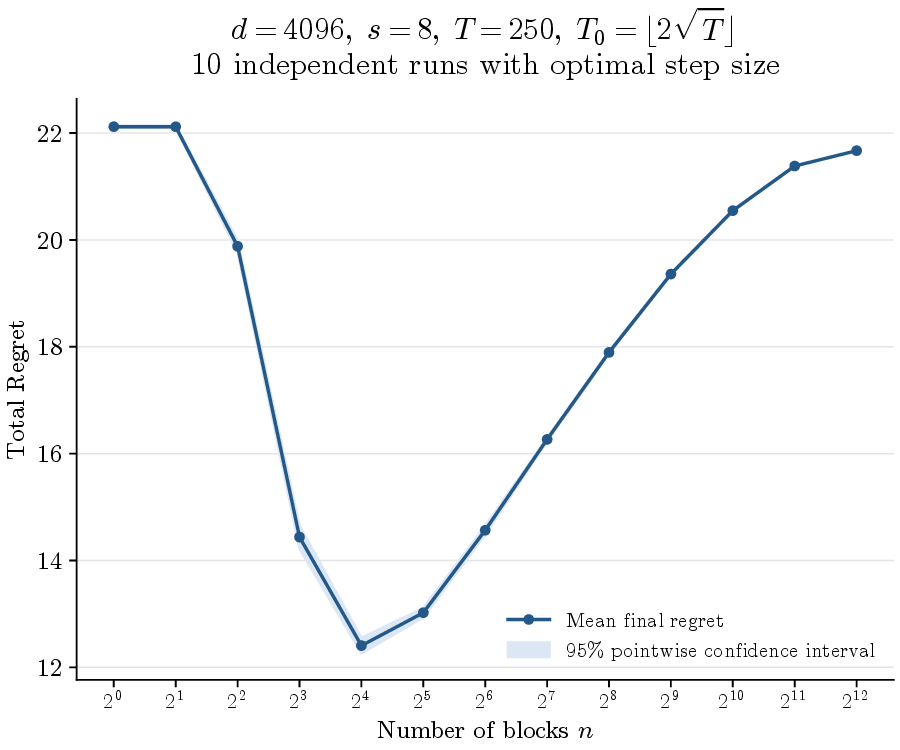}
    \end{minipage}
    \hfill
    \begin{minipage}[c]{0.40\linewidth}
        \caption{A numerical experiment to show the benefit in regret incurred using \gls{omd} setups with different block norms, in $d = 4096$ dimensions.
            The x-axis shows $13 = 1 + \log_{2} d$ random block norms, over the $d$-dimensional simplex.
            The y-axis shows the regret incurred at the end of $T = 250$ time steps, with minimum regret achieved at an intermediate geometry between Euclidean/\gls{opgd} ($1$ block; left end-point) and the entropic/\gls{oeg} geometry ($d$ blocks, right end-point).
            See \cref{sec: experiment-details} for details.}
        \label{fig: simplex-regret-vs-block-norms-rate-agnostic}
    \end{minipage}
\end{figure}

\paragraph{Randomized block norms.}
We make progress towards these questions through the family of \emph{randomized block norm} mirror maps, parametrized by integer $n \in \{1, 2, \dots, d\}$ that denotes the `number of blocks'.
Given $n$, the $d$ coordinates are partitioned into $n$ random blocks $B_1, \dots, B_n \subseteq [d]$ of near-equal size, chosen uniformly at random.
The corresponding \emph{block norm} $\|x\|_{[n]} := \sum_{j = 1}^n \|x_{B_j}\|_2$ is defined as the sum of the $L_2$ norms of the vectors $x_{B_j} \in \R^{B_j}$ restricted to individual blocks $j = 1, \dots, n$.

Ben-Tal and Nemirovski\textquotesingle s lecture notes \cite{ben-tal_nemirovski_lecture_notes_2025} describe block norms with fixed, \emph{non-random} blocks for \emph{offline} mirror descent, and construct corresponding mirror maps $h_n$ that are $1$-strongly convex with respect to the block norm inside the unit norm ball $B_{[n]} := \{x \in \R^d: \|x\|_{[n]} \le 1\}$; see \cref{sec: block-norms-preliminaries}.

For $n = 1$ block, there is no randomization, and this recovers the $L_2$ norm or the Euclidean geometry/\gls{opgd}.
For $n = d$, similarly, we recover the $L_1$ norm or the entropic geometry/\gls{oeg}\footnote{Technically, \gls{oeg} is only defined over the probability simplex $\simplex_d$ and not the entire $L_1$ unit norm ball $\{x: \|x\|_1 \le 1\}$. The case $n = d$ provides a proxy for \gls{oeg} on the simplex, with a matching regret upper bound up to absolute constants, and extends to the $L_1$ unit norm ball.
Throughout this work, we use `\gls{oeg}/entropic geometry' and `$n = d$' interchangeably; see \cref{sec: block-norms-preliminaries}.}.
Thus, randomized block norms interpolate between the two classical geometries and provide additional flexibility, and -- as we show next -- often achieve smaller regret than the two classical geometries.
\cref{fig: simplex-regret-vs-block-norms-rate-agnostic} illustrates this for sparse losses in a synthetic experiment over the probability simplex.

\subsection{Overview of results}\label{sec: results-and-techniques}

\begin{table}[t]
    \centering
    \caption{Factor improvements in regret upper bounds over the bounded $s$-sparse loss class, using OMD with the optimal random block norm, compared with the better of  Euclidean and entropic geometries.
    Each entry $g$ in the last column asserts $\gain(\convexbody,s)=\widetilde{\Omega}(g)$, with lower order terms in dimension $d$ omitted.
    For nonzero $w \in \R_{> 0}^d$, coordinates are ordered $w_1\ge\cdots\ge w_d$; $w_{\le s}$ contains the largest $s$ entries and $w_{>s}$ the remaining entries.
    General formulas are stated for $s \ge 3$.}
    \label{tab: gains}
    \setlength{\aboverulesep}{0pt} \setlength{\belowrulesep}{0pt} \setlength{\cmidrulewidth}{\arrayrulewidth}
    \begin{tabular}{|cc|c|c|}
        \hline
        \multicolumn{2}{|c|}{Convex body} &
        Sparsity $s$ &
        Gain using block norms \\
        \hline\hline

        \multicolumn{2}{|c|}{
            \multirow{2}{*}{
                \begin{tabular}[c]{@{}c@{}}
                    Probability simplex\\
                    $\simplex_d=\{x\ge0:\sum_i x_i=1\}$
                \end{tabular}
            }
        } &
        \rule[-1.5ex]{0pt}{5ex}Arbitrary &
        $\displaystyle
        \tfrac{\min\{\sqrt s,\sqrt{\ln d}\}}{\ln s}$ \\
        \cmidrule{3-4}
        \multicolumn{2}{|c|}{} &
        $\ln d$ &
        $\sqrt{\ln d}$ \\
        \hline

        \multicolumn{1}{|c|}{
            \multirow{2}{*}{
                \begin{tabular}[c]{@{}c@{}}
                    $L_p$ norm ball\\
                    $\{x:\sum_i |x_i|^p\le1\}$
                \end{tabular}
            }
        } &
        $p\in[1,2]$ &
        \rule[-1.5ex]{0pt}{5ex}Arbitrary &
        $\displaystyle
        \tfrac{\min\{\sqrt s,d^{1-1/p}\}}
        {s^{1-1/p}}$ \\
        \cmidrule{2-4}
        \multicolumn{1}{|c|}{} &
        $p=4/3$ &
        $\sqrt d$ &
        $d^{1/8}$ \\
        \hline

        \multicolumn{1}{|c|}{
            \multirow{2}{*}{
                \begin{tabular}[c]{@{}c@{}}
                    Axis-aligned ellipsoid\\
                    $\{x:\sum_i x_i^2/w_i^2\le1\}$
                \end{tabular}
            }
        } &
        $w\in\R_{>0}^d$ &
        \rule[-1.5ex]{0pt}{5ex}Arbitrary &
        $
        \tfrac{\min\{w_1\sqrt s,\|w\|_2\}}
        {\|w_{\le s}\|_2}$ \\
        \cmidrule{2-4}
        \multicolumn{1}{|c|}{} &
        \begin{tabular}[c]{@{}c@{}}
            $w_1\asymp d^{-1/4}$\\
            $w_j\asymp d^{-1/2}$, $j\ge2$
        \end{tabular} &
        $\sqrt d$ &
        $d^{1/4}$ \\
        \hline

        \multicolumn{1}{|c|}{
            \multirow{2}{*}{
                \begin{tabular}[c]{@{}c@{}}
                    Box\\
                    $\{x:|x_i|\le w_i\ \forall i\}$
                \end{tabular}
            }
        } &
        $w\in \R_{> 0}^d$ &
        \rule[-1.5ex]{0pt}{5ex}Arbitrary &
        $\displaystyle
        \tfrac{\min\{\|w\|_2\sqrt s,\|w\|_1\}}
        {\|w_{\le s}\|_1+\sqrt s\,\|w_{>s}\|_2}$ \\
        \cmidrule{2-4}
        \multicolumn{1}{|c|}{} &
        \begin{tabular}[c]{@{}c@{}}
            $w_1\asymp d^{-1/4}$\\
            $w_j\asymp d^{-1}$, $j\ge2$
        \end{tabular} &
        $\sqrt d$ &
        $d^{1/4}$ \\
        \hline

        \multicolumn{1}{|c|}{
            \multirow{2}{*}{
                \begin{tabular}[c]{@{}c@{}}
                    Sum of $L_1$ and $L_\infty$ balls\\
                    $\alpha_1B_1+\alpha_\infty B_\infty$
                \end{tabular}
            }
        } &
        $\alpha_1,\alpha_\infty>0$ &
        \rule[-1.5ex]{0pt}{5ex}Arbitrary &
        $\displaystyle
        \tfrac{\min\{(\alpha_1+\alpha_\infty\sqrt d)\sqrt s,\,
            \alpha_1+\alpha_\infty d\}}
        {\alpha_1+\alpha_\infty\sqrt{sd}}$ \\
        \cmidrule{2-4}
        \multicolumn{1}{|c|}{} &
        \begin{tabular}[c]{@{}c@{}}
            $\alpha_1\asymp d^{-1/4}$\\
            $\alpha_\infty\asymp d^{-1}$
        \end{tabular} &
        $\sqrt d$ &
        $d^{1/4}$ \\
        \hline

        \multicolumn{1}{|c|}{
            \multirow{2}{*}{
                \begin{tabular}[c]{@{}c@{}}
                    Sum of $L_1$ and $L_2$ balls\\
                    $\alpha_1B_1+\alpha_2B_2$
                \end{tabular}
            }
        } &
        $\alpha_1,\alpha_2>0$ &
        \rule[-1.5ex]{0pt}{5ex}Arbitrary &
        $\displaystyle
        \tfrac{\min\{(\alpha_1+\alpha_2)\sqrt s,\,
            \alpha_1+\alpha_2\sqrt d\}}
        {\alpha_1+\alpha_2\sqrt s}$ \\
        \cmidrule{2-4}
        \multicolumn{1}{|c|}{} &
        \begin{tabular}[c]{@{}c@{}}
            $\alpha_1\asymp d^{-1/4}$\\
            $\alpha_2\asymp d^{-1/2}$
        \end{tabular} &
        $\sqrt d$ &
        $d^{1/4}$ \\
        \hline

        \multicolumn{1}{|c|}{
            \multirow{2}{*}{
                \begin{tabular}[c]{@{}c@{}}
                    Segment plus $L_2$ ball\\
                    $\alpha_0[-\mathbf e_1,\mathbf e_1]+\alpha_2B_2$
                \end{tabular}
            }
        } &
        $\alpha_0,\alpha_2>0$ &
        \rule[-1.5ex]{0pt}{5ex}Arbitrary &
        $\displaystyle
        \tfrac{\min\{(\alpha_0+\alpha_2)\sqrt s,\,
            \alpha_0+\alpha_2\sqrt d\}}
        {\alpha_0+\alpha_2\sqrt s}$ \\
        \cmidrule{2-4}
        \multicolumn{1}{|c|}{} &
        \begin{tabular}[c]{@{}c@{}}
            $\alpha_0\asymp d^{-1/4}$\\
            $\alpha_2\asymp d^{-1/2}$
        \end{tabular} &
        $\sqrt d$ &
        $d^{1/4}$ \\
        \hline
    \end{tabular}
\end{table}

Randomized block norms give a general recipe for sparse losses: for each $n$, use the normalized mirror map $h_{n,\rho_n(\convexbody)}$, which is $1$-strongly convex with respect to the sampled block norm on $\convexbody$.
Here $\rho_n(\convexbody)$ bounds the radius over all equal-size $n$-partitions; see \cref{sec: block-norms-preliminaries}.
For $n\in\{1,2,\dots,d\}$, we denote by $D_n:=D_n(\convexbody)$ the corresponding root-mean-square Bregman diameter.

When sparsity $s$ is known, the per-round gradient bound satisfies $G_n\le G_n(s)$, where the uniform bound $G_n(s)$ is estimated in \cref{lem: sparse-vector-block-variance-bound}.
Invoking the regret upper bound,
\[
    \tfrac{\E[\regret_n(T)]}{2\sqrt T}\le D_nG_n\le D_n G_n(s).
\]
Optimizing the product $\psi_n := D_n G_n$ over $n$ upper bounds the regret using the best randomized block norm.
In particular, since $n = 1$ corresponds to the Euclidean geometry and $n = d$ corresponds to the entropic geometry, this regret bound is always smaller than the best of \gls{opgd} and \gls{oeg}.
Formally, we call this factor improvement in regret bound the \emph{gain} of block norms:

\paragraph{Gains using block norms.}
In our first result, we show in \cref{sec: gains-using-block-norms} that this improvement over \gls{opgd} and \gls{oeg} can be non-trivial for many common convex bodies.
To formally state our result, we define the \emph{gain} of block norms:
\begin{definition}[Block norm gain]
    Given an \gls{oco} instance with feasible set $\convexbody \subseteq \R^d$ and $s$-sparse loss function gradients, and time horizon $T$, the \emph{gain} of block norms, denoted $\gain := \gain(\convexbody, s)$ is the improvement in regret upper bound using block norms as compared to the best of Euclidean and entropic geometries:
    \begin{equation}\label{eqn: gain-def}
        \gain(\convexbody, s) := \frac{\min\{2 D_\euc G_\euc \sqrt{T}, 2 D_{\ent} G_{\ent} \sqrt{T}\}}{\min_{n = 1}^d \{2 D_n G_n \sqrt{T}\}} = \frac{\min\{D_\euc G_\euc, D_\ent G_\ent\}}{\min_{n = 1}^d D_n G_n}.
    \end{equation}
\end{definition}

As discussed, gain for any bounded $\convexbody$ and sparsity $s$ is at least $1$.
Note that since it is a comparison between regret upper bounds of various \gls{omd} algorithms, it is independent of the time horizon $T$.
Our first result shows that the gain can be as large as $\widetilde{\Omega}(d^{1/4})$ for convex sets in $\R^d$.
This includes standard convex bodies such as ellipsoids, boxes, and simple convex combinations; \cref{tab: gains} lists these gains explicitly.

\begin{restatable}[Regret bound improvements for sparse losses]{theorem}{RegretUpperBoundImprovements}\label{thm: upper-bound-improvements}
    The bounded convex sets $\convexbody_d \subseteq \R^d$ and corresponding sparsities $s_d \in [d]$ yield gains $\gain(\convexbody, s)$ as listed in \cref{tab: gains}.

    In particular, OMD with an appropriate block norm achieves a regret upper bound within a factor $\widetilde{O}(d^{-1/4})$ of the corresponding upper bounds for both \gls{opgd} and \gls{oeg} for certain families of convex sets $\convexbody_d \subseteq \R^d$ and corresponding sparsities $s_d \in [d]$.
\end{restatable}

To prove gains in \cref{tab: gains}, we need to (1) lower bound Euclidean and entropic diameters ($D_\euc, D_\ent$ resp.) as well as gradient bounds ($G_\euc, G_\ent$ resp.), and (2) upper bound $D_n$ and $G_n$ for the $n$th block norm.
Since blocks are chosen uniformly at random, we bound $G_n$ for $s$-sparse losses in \cref{lem: sparse-vector-block-variance-bound} using standard concentration inequalities.
Bounding the root-mean-square diameter $D_n$ for given starting point $\var{x}{1} \in \convexbody$ is more involved; we show in \cref{lem: block-norm-diameter-bound-general} that $D_n=O(\rho_n(\convexbody)\sqrt{\ln n})$, reducing this to bounding the block norm radius of $\convexbody$.

The gains can be extended to arbitrary convex sets $\convexbody$ by `sandwiching' it between scaled copies of a centrally symmetric body $\convexbody_0$ (i.e., for all $x \in \convexbody_0$, $-x$ is also in $\convexbody_0$):

\begin{restatable}[Sandwich theorem for gains]{corollary}{SandwichGain}\label{cor: sandwich-gain}
    Consider nonempty bounded convex sets $\convexbody,\convexbody_0\subseteq \R^d$ such that  $\convexbody_0 \subseteq \convexbody\subseteq\beta\convexbody_0$ for some $\beta \ge 1$.
    If $\convexbody_0$ is centrally symmetric about the origin ($\convexbody_0=-\convexbody_0$), then for all $s\in[d]$,
    \[
        \gain(\convexbody,s)\ge\frac{1}{\beta}\,\gain(\convexbody_0,s).
    \]
    for any given starting point $\var{x}{1}\in\convexbody_0$.
\end{restatable}

\paragraph{Mirror maps induced by geometry of feasible set.}
We show that the optimal mirror map depends not only on the geometry of the feasible set $\convexbody$, but also on the gradient set $\mathcal{G}$.
Any feasible set $\convexbody$ that is convex and centrally symmetric induces a norm, and any mirror map that is strongly convex with respect to the norm induced by $\convexbody$ is another natural choice for \gls{omd}.
For example, Cheng \emph{et al.} \cite{cheng2019geometry} show that when $\convexbody$ is the $L_p$ norm ball with $1 \le p \le 2$ and the gradients are restricted to the dual $L_q$ norm ball (where $1/p + 1/q = 1$), an appropriately chosen norm-based mirror map achieves minimax-optimal regret rates for \gls{oco}.
In \cref{sec: induced-mirror-maps-are-suboptimal}, we show that $L_p$ mirror maps with $p>1$ can have suboptimal regret upper bounds for sparse gradients on the $L_p$ norm ball, and that a random block norm achieves asymptotically better regret bound.

\paragraph{Instances where gains are realized.}
\cref{thm: upper-bound-improvements} and \cref{tab: gains} show that block norms can improve the \emph{worst-case} regret upper bound \cref{eqn: regret-bound-general} over Euclidean and entropic mirror maps by large factors.
In \cref{sec: instances-with-realized-gains}, we construct explicit \gls{oco} instances where these improvements are realized: one for the probability simplex, and another for a convex hull of the probability simplex and a single point of the form $a \mathbf{1}_d$.
Instances where block norms beat either the Euclidean or the entropic geometry follow using standard lower-bound constructions \cite{hazan_introduction_2016, mourtada2019optimality}; our main technical contribution lies in finding instances where the realized regret is \emph{simultaneously} poor for both Euclidean and entropic mirror maps as compared to the optimal block norm.

\begin{restatable}[Realized regret gains for sparse losses]{theorem}{regretImprovementsBlockNorm}\label{thm: regret-improvement-using-block-norms} \phantom{ }
    \begin{enumerate}
        \item There exists a family of online convex optimization (\gls{oco}) instances over the polytopes $\convexbody_d=\mathrm{conv}(\simplex_d \cup \{d^{-3/4}\mathbf{1}_d\}) \subseteq \R^d$ with $s= \sqrt d$-sparse loss functions where,
        \[
            \min\{\E [\regret_{\euc}(T)], \E[\regret_d(T)]\} = \Omega\left(\frac{d^{1/4}}{\ln d}\right) \ \E[\regret_s(T)].
        \]
        That is, \gls{omd} with the mirror map $h_{s,\rho_s(\convexbody_d)}$ corresponding to the $s$th block norm improves the regret by a factor $\Omega(d^{1/4}/\ln d)$ over the best of online projected gradient descent (\gls{opgd}) and OMD$_d$ with standard step sizes.

        \item There exists a family of online convex optimization (\gls{oco}) instances over the probability simplices $\simplex_d$ with $s= \ln d$-sparse loss functions, where
        \[
            \min\left\{\E[\regret_\euc(T)], \E[\regret_\ent(T)]\right\} = \Omega\left(\tfrac{\sqrt{\ln d}}{\ln \ln d}\right) \ \E[\regret_{\lfloor\ln d\rfloor}(T)].
        \]
        That is, OMD with the $\lfloor\ln d\rfloor$th block norm improves the regret by a factor $\Omega\left(\tfrac{\sqrt{\ln d}}{\ln \ln d}\right)$ over the best of \gls{opgd} and \gls{oeg} with standard step sizes.
    \end{enumerate}
\end{restatable}

For each constructed instance, the regret for block norms is upper bounded using \cref{eqn: regret-bound-general} (as before).
However, since this is an upper bound, it cannot be used to lower bound the actual regret for \gls{opgd} and \gls{oeg}.
Instead, we show that for these instances, the \gls{opgd}/\gls{oeg} iterates $\var{x}{t}$ stay `far enough' from the optimal $x^*$, incurring a `large enough' regret at any time step $t \in [T]$.
For both examples and both algorithms, we require separate, delicately constructed notions of what is `far enough' or `large enough'.

\paragraph{Learning optimal block size for unknown sparsity.}
In the above results, the choice of the optimal number of blocks depends on sparsity $s$.
In \cref{sec: learning-optimal-mirror-map}, we consider settings where the sparsity $s$ is unknown in advance.

First, in \cref{sec: alternating-mirror-maps-is-suboptimal}, we construct an example where using the standard step size in the \gls{oeg} and \gls{opgd} setups, switching the mirror maps in each iteration can in fact cause the algorithm to converge to a suboptimal point in the domain, resulting in a linear regret.
This is surprising as any online mirror descent setup must decrease a certain potential function across each iteration, and therefore one might expect that alternating between two mirror maps does not affect convergence.
However, alternately changing the mirror maps breaks this monotonicity due to the interaction of divergences:

\begin{restatable}[Linear regret when alternating between mirror maps]{theorem}{alternatingMirrorMaps}\label{thm: alternating-mirror-maps-is-suboptimal}
    There exists a family of linear loss functions over the probability simplex $\simplex_2 := \{x \in \R^2: x_1 + x_2 = 1, x \ge 0\}$ in two dimensions such that the algorithm that starts at $x^{(1)} = (1/2, 1/2)$ and alternates between the updates for online projected gradient descent (\gls{opgd}) and online exponentiated gradient (\gls{oeg}) at each time step incurs regret $\Omega(T)$, irrespective of any fixed step sizes for \gls{opgd} and \gls{oeg}.
\end{restatable}

Next, in \cref{sec: mwu-over-mirror-maps}, we design an adaptive algorithm that can tune to the best block norm in the portfolio, by using a Hedge meta-algorithm over different mirror maps \cite{hazan_introduction_2016}.
Our algorithm is not restricted to block norms, and applies generally to any collection of mirror maps.
Aggregation over learning rates is used in MetaGrad \cite{van2016metagrad}; other approaches to parameter adaptation include AdaGrad \cite{duchi2011adaptive}, AdaHedge \cite{erven2011adaptive}, and coin-betting algorithms \cite{Orabona2016}.

\begin{restatable}[Learning mirror maps online]{theorem}{combiningMirrorMaps}\label{thm: mirror-weights}
    There exists an online convex optimization (\gls{oco}) algorithm that given a convex body $\convexbody \subseteq \R^d$, a set $M = \{m_1, \ldots, m_N\}$ of $|M| = N$ mirror maps on $\convexbody$, and any sequence of online loss functions $f^{(1)}, \ldots, f^{(T)}$ such that $\max_{t \in [T], x, z \in \convexbody} f^{(t)}(x) - f^{(t)}(z) \le \rho$, achieves regret at most
    \[
        \regret(T) \le \min_{\ell \in [N]} \regret_{m_\ell}(T) + 2 \rho \sqrt{T \ln N},
    \]
    as long as $T \ge \ln N$.
\end{restatable}

To apply this to random block norms, we first observe that there are effectively $1 + \log_2 d$ block norms corresponding to $n \in \{2^0, 2^1, \dots, 2^{\log_2 d}\}$, allowing constant factor increase in regret.
Therefore, using $|M| = O(\log d)$ in the above theorem, we get the following corollary:

\begin{restatable}[Learning block norms online]{corollary}{combiningBlockNorms}\label{thm: mirror-weights-norms}
    There exists an algorithm for online convex optimization (\gls{oco})  that given a convex body $\convexbody \subseteq \{x \in \R^d: \|x\|_1 \le 1\}$ that lies within the $L_1$ norm ball in $\R^d$, and any sequence of online loss functions $f^{(1)}, \ldots, f^{(T)}$ with bounded gradients achieves regret at most
    \[
        \tfrac{\E[\regret(T)]}{\sqrt{T}} = O(\sqrt{\ln \ln d}) \cdot  \min_{n \in [d]} D_n G_n,
    \]
    as long as $T \ge 4 \ln \ln d$, where $G_n = \max_{t\in[T]}\big(\E[\sup_{x\in\convexbody}\|\nabla f^{(t)}(x)\|_{[n]}^{*2}]\big)^{1/2}$ is the gradient bound for the $n$th random block norm.
\end{restatable}

\paragraph{Extension to stochastic losses.}
We show that our results extend naturally to \emph{unbounded} stochastic losses, as long as the second moments of the gradients are bounded.
This relies on the well-known fact that the \gls{omd} regret bound \cref{eqn: regret-bound-general} also holds for stochastic losses with bounded second moments (see \cref{sec: mirror-maps-omd}; \cite{bubeck_convex_2015}).
Accordingly, for each loss function $f$, define
\[
    \sigma_f^2 := \max_{x \in \convexbody} \E\Big[\max_{i \in [d]} \left(\nabla_i f(x)\right)^2\Big], \qquad \sigma^2 := \max_{\mathrm{loss \ function\ } f} \sigma_f^2,
\]
and suppose that the random loss gradients are always $s$-sparse.
Then, as long as the losses are generated independently of the random partition (which is drawn once at $t = 0$), conditioning on the realization of the gradient and applying \cref{lem: sparse-vector-block-variance-bound} to its $0$-$1$ support vector gives $\E\big[\big(\|\nabla f(x)\|^*_{[n]}\big)^2\big] \le 9\sigma^2 \max\left\{\tfrac{s}{n}, \ln n\right\}$ for every fixed $x \in \convexbody$, and consequently our regret upper bounds using block norms continue to hold, with the expectation now taken over both the random partition and the realization of the losses.

\subsection{Related work}\label{sec: mirror-maps-related-work}

The study of \gls{oco} and its canonical algorithms builds on early work on weighted-majority and exponentiated-gradient methods \cite{littlestone1994weighted, kivinen1997exponentiated, cesa1997use}, Zinkevich\textquotesingle s analysis of Online Projected Gradient Descent \cite{zinkevich_online_2003}, and the mirror-descent framework \cite{nemirovski1983problem}; see \cite{shalev-shwartz_online_2012, hazan_introduction_2016} for accounts of these developments.
For most standard convex sets, these two mirror maps -- Euclidean ($h(x) = \frac{1}{2}\|x\|_2^2$) and entropic ($h(x) = \sum_i x_i \log x_i$, or an equivalent proxy) -- have long been viewed as the two complementary geometries: Euclidean for the hypercube and entropic for the probability simplex.
Both achieve minimax-optimal regret of order $O(f(d)\sqrt{T})$, with dimension dependence $f(d)$ varying by the setting.

It is known that other mirror maps can achieve better asymptotic regret rates in certain \gls{oco} instances.
For example, Kwon and Perchet \cite{kwon2016gains} show that using the appropriate $L_p$ norm mirror map achieves regret $O(\sqrt{(\ln s) T})$ for sparse gains over the probability simplex; the same bound remains valid for sparse losses, improving over the standard \gls{opgd} bound $O(\sqrt{sT})$ and \gls{oeg} bound $O(\sqrt{(\ln d)T})$ when $s = \ln d$.
For nonnegative sparse losses, they also establish the sharper minimax rate $\Theta(\sqrt{Ts\ln(d)/d})$ in their asymptotic regime.
To the best of our knowledge, we are the first to show polynomial in dimension improvement in regret over both \gls{opgd} and \gls{oeg} (or an equivalent proxy).
The lower bounds in \cite{kwon2016gains} distinguish the minimax rates for sparse gains and sparse nonnegative losses; we construct a single instance where both Euclidean (\gls{opgd}) and entropic (\gls{oeg}) mirror maps with standard step sizes are simultaneously suboptimal.

In contrast, we establish the first explicit, polynomial-in-dimension separation in realized regret between a block-norm mirror map and \emph{both} the Euclidean mirror map and the entropy proxy with standard step sizes, simultaneously on a single instance over a simple convex polytope not contained in the simplex, to the best of our knowledge.
We construct a natural convex body, $P = \mathrm{conv}(\simplex_d, A\mathbf{1})$, and an oblivious linear-loss sequence for which OPGD (Euclidean) and the entropy proxy (mirror map for $d$th block norm) both suffer regret scaling as $\widetilde \Omega(d^{1/4}\sqrt{T})$, while OMD with a block-norm mirror map tailored to the sparsity structure achieves $\widetilde{O}(\sqrt{T})$.
Individual block norms are described for (offline) Mirror Descent in Ben-Tal and Nemirovski\textquotesingle s lecture notes \cite{ben-tal_nemirovski_lecture_notes_2025} as alternatives to standard Euclidean geometry.
Mixed-norm regularization has also been studied within composite objective mirror descent \cite{duchi2010composite}; our focus is instead on regret improvements obtained by choosing block-norm mirror maps.

AdaGrad \cite{duchi2011adaptive} and its variants restrict to a general Euclidean mirror map, but update it using a (pseudo) second-order update.
In particular, diagonal AdaGrad achieves improved regret when non-zero gradient is concentrated in a small number of coordinates.
Our block-norm methods exploit a different structure: sparsity of each loss gradient.
For example, consider linear losses with binary $s$-sparse gradients on the probability simplex where each coordinate is active on exactly $sT/d$ rounds.
The standard diagonal AdaGrad regret bound \cite[Corollary~1]{duchi2011adaptive} evaluates to $O(\sqrt{dsT})$, whereas our bound  using OMD$_s$ is $O(\sqrt{T}\ln s)$.

The problem of finding the optimal mirror map for a given \gls{oco} instance has also been considered: Gatmiry \emph{et al.} \cite{gatmiry2025computing} give a construction that, given a convex, centrally symmetric feasible set $\convexbody$ and a convex, centrally symmetric gradient set $\mathcal{G}$, obtains a near-optimal regularizer as the solution to a finite optimization problem.
This optimization problem can be solved (in exponential time), yielding a regularizer whose Follow-the-Regularized-Leader algorithm has worst-case regret within a universal constant factor of the minimax-optimal rate.
Our focus is on explicit, efficiently specified block-norm mirror maps and their regret improvements for sparse losses, including on feasible sets that are not centrally symmetric.
Cheng \emph{et al.} \cite{cheng2019geometry} characterize when Euclidean methods (mirror descent with any quadratic mirror map, including diagonal preconditioners) are optimal for given convex set and gradient set, whereas we quantify improvements for sparse gradient sets using block norms.

Parameter-free and adaptive algorithms for online convex optimization use different mechanisms: AdaHedge \cite{erven2011adaptive} adapts the learning rate using the mixability gap, MetaGrad \cite{van2016metagrad} aggregates over learning rates using surrogate losses and adaptive quadratic geometry, and the coin-betting algorithms of Orabona and P\'al \cite{Orabona2016} avoid explicit learning-rate tuning.
These methods adapt within their respective algorithmic frameworks, rather than explicitly searching over a prescribed portfolio of distinct mirror maps.
In contrast, our method uses Hedge algorithm across multiple \emph{distinct} mirror maps, enabling adaptation not only over step sizes but also over geometries.

\section{Preliminaries}\label{sec: block-norms-preliminaries}

In this section, we give preliminaries and establish some notation.
The \gls{omd} algorithm is defined by the update rule in \cref{eqn: omd-update}, which describes the choice of played point $x^{(t + 1)}$; see \cref{sec: mirror-maps-omd} for details on \gls{omd}.

\subsection{Block norms, OMD, and randomization}

Given arbitrary blocks $\mathcal{B} = \{B_1, \dots, B_n\}$ that partition the coordinate set $[d]$, then the corresponding block norm is defined as
\begin{equation}\label{eqn: block-norm-definition}
    \|x\|_{[\mathcal{B}]} := \sum_{j = 1}^n \|x_{B_j}\|_2 = \sum_{j = 1}^n \sqrt{\sum_{i \in B_j} x_i^2}.
\end{equation}
The corresponding norm ball of radius $\rho \ge 0$ is denoted $N_{\mathcal{B}, \rho} := \{x \in \R^d: \|x\|_{[\mathcal{B}]} \le \rho\}$.

It is easy to check that the dual norm is the \emph{maximum} $L_2$ norm over blocks:
\begin{equation}\label{eqn: block-norm-dual}
    \|g\|^*_{[\mathcal{B}]} := \max_{j = 1}^n \|g_{B_j}\|_2.
\end{equation}

Ben-Tal and Nemirovski\textquotesingle s lecture notes \cite{ben-tal_nemirovski_lecture_notes_2025} give a mirror map $h_{\mathcal{B}}$ that is $1$-strongly convex with respect to the block norm $\|x\|_{[\mathcal{B}]} := \sum_{j = 1}^n \|x_{B_j}\|_2$ in its unit norm ball:
\begin{theorem}[\cite{ben-tal_nemirovski_lecture_notes_2025}, \S 5.3.3.3, Eq.~(5.3.9)]\label{thm: block-norms-dgf}
    Given positive integers $n \le d$, define
    \begin{align*}
        \gamma_n =
        \begin{cases}
            1 & \text{if} \ n = 1, \\
            \frac{1}{2} & \text{if} \ n = 2, \\
            \frac{1}{e\ln n} & \text{if} \ n > 2,
        \end{cases}
        && p_n &=
        \begin{cases}
            2 & \text{if} \ n \le 2, \\
            1 + \frac{1}{\ln n} & \text{if} \ n > 2.
        \end{cases}
    \end{align*}
    For $n$ blocks $\mathcal{B} = \{B_1, \ldots, B_n\}$ that partition $[d]$, define
    \begin{equation}\label{eqn: block-norm-strongly-convex-function}
        h_{\mathcal{B}}(x) := \frac{1}{\gamma_n p_n}\sum_{j = 1}^n \|x_{B_j}\|_2^{p_n}.
    \end{equation}
    Then $h_{\mathcal{B}}$ is $1$-strongly convex with respect to the block norm corresponding to $\mathcal{B}$ over the unit norm ball $N_{\mathcal{B}, 1} = \{x: \|x\|_{[\mathcal{B}]} \le 1\}$.
\end{theorem}

\begin{corollary}\label{cor: scaled-mirror-maps}
    Consider $n$ blocks $\mathcal{B} = \{B_1, \ldots, B_n\}$ that partition $[d]$ and a non-empty set $\convexbody \subseteq N_{\mathcal{B}, \rho}$ that lies in the corresponding norm ball of radius $\rho > 0$.
    Then, the mirror map $h_{\mathcal{B}, \rho}(x) := \rho^2 h_{\mathcal{B}}\big(\tfrac{x}{\rho}\big)$ is $1$-strongly convex with respect to $\|\cdot\|_{[\mathcal{B}]}$ over $\convexbody$.
\end{corollary}

\begin{proof}
    For $x,z\in\convexbody\subseteq N_{\mathcal B,\rho}$, the chain rule and strong convexity on $N_{\mathcal B,1}$ give
    \begin{align*}
        B_{h_{\mathcal B,\rho}}(x\|z)
        &=\rho^2 B_{h_{\mathcal B}}(x/\rho\|z/\rho) \ge\frac{\rho^2}{2}\|(x-z)/\rho\|_{[\mathcal B]}^2
        =\frac12\|x-z\|_{[\mathcal B]}^2. \qedhere
    \end{align*}
\end{proof}

For a nonempty convex set $\convexbody \subseteq \R^d$, define $\rho_{n}(\convexbody)$ as the maximum radius of $\convexbody$ under any arbitrary equal-size\footnote{For simplicity, we use the term `equal-size' even when $d$ is not divisible by $n$, where `equal-size' is shorthand for blocks with size difference at most $1$.}
block norm:
\begin{equation}\label{eqn: block-norm-radius}
    \rho_n(\convexbody) := \max_{\substack{\textup{equal-size blocks} \\ \mathcal{B} = (B_1, \dots, B_n)}} \sup_{x \in \convexbody} \|x\|_{[\mathcal{B}]}.
\end{equation}
Then, for any equal-size blocks, $\mathcal{B} := \{B_1, \ldots, B_n\}$, the above result implies that $h_{\mathcal{B}, \rho_n(\convexbody)}$ is $1$-strongly convex with respect to $\|\cdot\|_{[\mathcal{B}]}$ on $\convexbody$.

\paragraph{OMD with random block norms.}
We develop a random variant of \gls{omd} with block norms as follows: for $n \le d$, consider an equal-size partition $\mathcal{B}$ of coordinates $[d]$, chosen uniformly at random \emph{a priori} (before the first round), independently of the loss sequence, and retained throughout the algorithm.
We denote by $\|\cdot\|_{[n]}$ the corresponding (random) block norm.
Given radius $\rho > 0$, we denote the corresponding random mirror map $h_{n, \rho} = h_{\mathcal{B}, \rho}$.

Given \gls{oco} instance on convex set $\convexbody \subseteq \R^d$ and starting point $\var{x}{1}$, consider the \gls{omd} algorithm with the mirror map $h_{n, \rho}$ for $\rho := \rho_{n}(\convexbody)$.
This \gls{omd} variant is denoted \gls{omd}$_n$, and its regret is denoted $\regret_n$.
The corresponding Bregman divergence $B_{h_{n, \rho}}(\cdot \|\cdot)$ is denoted by $B_{n, \rho}(\cdot \|\cdot)$ for brevity.

Define the root-mean-square diameter and per-round gradient bound by
\begin{align}
    D_n &:= \left(\E_{\mathcal B}\left[\max_{z\in\convexbody}
    B_{n,\rho}(z\|x^{(1)})\right]\right)^{1/2}, \label{eqn: defn-block-norm-diameter}\\
    G_n&:=\max_{t\in[T]}\left(\E_{\mathcal B}\left[
    \sup_{x\in\convexbody}\|\nabla f^{(t)}(x)\|_{[n]}^{*2}\right]\right)^{1/2}. \label{eqn: defn-block-norm-gradient-bound}
\end{align}

Since $h_{n,\rho}$ is $1$-strongly convex with respect to $\|\cdot\|_{[n]}$ on $\convexbody$, the general \gls{omd} regret bound, with step size $\eta=D_n/(G_n\sqrt T)$, implies
\begin{equation}\label{eqn: block-norm-omd-regret-bound}
    \E[\regret_n(T)]
    \le \frac{D_n^2}{\eta}+\frac{\eta}{2}\sum_{t=1}^T
    \E_{\mathcal B}\!\left[\big(\|\nabla f^{(t)}(x^{(t)})\|_{[n]}^{*}\big)^{2}\right]
    \le 2D_nG_n\sqrt T.
\end{equation}
Replacing $D_n$ and $G_n$ throughout by positive deterministic upper bounds gives the corresponding regret guarantee.
The expectations above are over the random partition, for a fixed loss sequence.
Under \cref{assmp: loss-functions}, the uniform bounds established below also apply to random loss sequences independent of the partition, by conditioning on the losses and choosing the step size using these deterministic bounds.

\paragraph{Euclidean and entropic geometries.}

When the number of blocks $n = 1$ or $n = d$, we recover deterministic blocks.
For $n = 1$, we recover the Euclidean/$L_2$ geometry exactly, and the algorithm reduces to \gls{opgd}.
Therefore, we interchangeably use \gls{opgd} and \gls{omd}$_1$, $D_1$ and $D_\euc$, $G_1$ and $G_\euc$ etc.

Now consider $n = d$, where the block norm reduces to the $L_1$ norm.
On the probability simplex $\simplex_d := \{x \ge 0: \sum_i x_i = 1\}$, the corresponding algorithm OMD$_d$ is a proxy for \gls{oeg}: with starting point $\var{x}{1} = \tfrac{1}{d} \mathbf{1}_d$, its worst-case Bregman divergence from this starting point is $\Theta(\ln d)$, and its standard \gls{omd} regret bound agrees with that of \gls{oeg} up to a universal constant.
For brevity, we interchangeably use \gls{oeg} and \gls{omd}$_d$, $D_d$ and $D_\ent$, $G_d$ and $G_\ent$ etc.

\paragraph{Concentration for negatively associated random variables.}

We will use the following concentration inequality that generalizes concentration around mean results for independent random variables to the broader class of negatively associated random variables:

\begin{lemma}[Bernstein inequality \cite{BertailClemencon2019}, Theorem 2]\label{lem: bernstein}
    Let $X_1, \ldots, X_N$ be negatively associated zero-mean random variables with $|X_i| \le 1$ for all $i \in [N]$ with probability $1$.
    Let $z_1, \ldots, z_N$ be nonnegative constants.
    Then, for all $\delta > 0$,
    \begin{equation*}
        \Pr\Big(\sum_{i \in [N]} z_i X_i \ge \delta\Big) \le \exp\Big(-\frac{3\delta^2}{2\delta \max_{i \in [N]} z_i + 6 \sum_{i \in [N]}z_i^2 \E[X_i^2]}\Big).
    \end{equation*}
\end{lemma}

\section{Regret bounds for sparse losses}\label{sec: mirror-maps-better-regret}

In this section, we discuss regret improvements over Euclidean and entropic geometries using block norms.
First, we show that convex combinations of the corresponding mirror maps $h_\euc := \tfrac{1}{2}\|x\|_2^2$ and $h_\ent := \sum_i x_i \ln x_i$ do not improve regret over the simplex in \cref{app: convex-combination-is-not-useful}.
Then, we study block norms, and prove general gradient and diameter bounds in \cref{sec: gradient-and-diameter-bounds}.
In \cref{sec: gains-using-block-norms}, we prove \cref{thm: upper-bound-improvements} proving gains in regret bounds using block norms for common convex sets, and \cref{cor: sandwich-gain} which shows gains for convex sets sandwiched between other convex sets.
In \cref{sec: instances-with-realized-gains}, we prove \cref{thm: regret-improvement-using-block-norms} giving explicit \gls{oco} instances where large gains using block norms over the best of Euclidean and entropic geometries are realized.

\subsection{Convex combinations of classical geometries}\label{app: convex-combination-is-not-useful}

In this section, we consider mirror maps $h_\lambda := \lambda h_\euc + (1 - \lambda) h_\ent$ that are convex combinations of Euclidean ($h_\euc(x) := \tfrac{1}{2} \|x\|_2^2$) and entropic ($h_\ent(x) := \sum_{i \in [d]} x_i \ln x_i$) mirror maps.
We show that such mirror maps cannot yield better than a constant-factor improvement in regret bound over the best of \gls{opgd} and \gls{oeg} for $s$-sparse losses.

Recall that for the simplex $\simplex_d$, the Euclidean and entropic diameters are $D_\euc = \Theta(1)$ and $D_\ent := \Theta(\sqrt{\ln d})$ respectively, and the gradient bounds for $s$-sparse losses in the bounded box $\mathcal{G} := \{g \in [-1, 1]^d: \|g\|_0 \le s\}$ are $G_\euc = \max_{g \in \mathcal{G}} \|g\|_2 = \sqrt{s}$ and $G_\ent = \max_{g \in \mathcal{G}} \|g\|_\infty = 1$.

For any $\lambda \in [0, 1]$, the convex combination mirror $h_\lambda$ is $1$-strongly convex with respect to the composite norm $N_\lambda(x) := \sqrt{\lambda \|x\|_2^2 + (1 - \lambda) \|x\|^2_1}$.

Further, the diameter satisfies $D_\lambda^2 \ge \max(\lambda D_\euc^2, (1 - \lambda) D_\ent^2) \ge \tfrac{1}{2} (\lambda D_\euc^2 + (1 - \lambda) D_\ent^2)$.
The gradient bound for $s$-sparse losses is then expressed using the dual norm $N^*_\lambda(\cdot)$ as
\begin{equation}\label{eqn: gradient-bound-conv-comb}
    G_\lambda := \max_{g \in [-1, 1]^d: \|g\|_0 \le s} N^*_\lambda(g) = \sqrt{\tfrac{s}{\lambda + (1 - \lambda) s}}.
\end{equation}

Assuming this bound, we get
\begin{align*}
    2 D_\lambda^2 G_\lambda^2 &\ge \tfrac{\lambda D_\euc^2 s + (1 - \lambda) D_\ent^2 s}{\lambda + (1 - \lambda) s} = \tfrac{\lambda D_\euc^2 G_\euc^2 + s(1 - \lambda) D_\ent^2 G_\ent^2}{\lambda + s(1 - \lambda)} \ge \min\left(D_\euc^2 G_\euc^2, D_\ent^2 G_\ent^2\right),
\end{align*}
where the last inequality holds since $\tfrac{\alpha A + \beta B}{\alpha + \beta} \ge \min(A, B)$ for $A, B, \alpha, \beta \ge 0$ (here $\alpha = \lambda$ and $\beta = s(1 - \lambda)$).
It remains to prove \cref{eqn: gradient-bound-conv-comb}:

\begin{proof}[Proof of \cref{eqn: gradient-bound-conv-comb}]
    The case $g=0$ is immediate.
    Fix nonzero $g \in \mathcal{G}$ and let $T := \operatorname{supp}(g)$, so that $|T| \le s$ and $|g_i| \le 1$ for all $i \in T$.
    For $x \in \R^d$, write $x_T$ for the restriction of $x$ to the coordinates in $T$.
    Since $\|x_T\|_2 \le \|x\|_2$ and $\|x_T\|_1 \le \|x\|_1$, we have $N_\lambda(x_T) \le N_\lambda(x)$, and by H\"older's inequality $\langle g, x\rangle = \sum_{i \in T} g_i x_i \le \|x_T\|_1$.
    Therefore
    \[
        N^*_\lambda(g) = \sup_{x \ne 0} \frac{\langle g, x\rangle}{N_\lambda(x)}
        \le \sup_{x_T \ne 0} \frac{\|x_T\|_1}{N_\lambda(x_T)}
        = \sup_{y \in \R^T \setminus \{0\}} \frac{\|y\|_1}{\sqrt{\lambda \|y\|_2^2 + (1 - \lambda)\|y\|_1^2}}.
    \]
    By Cauchy-Schwarz, $\|y\|_1^2 \le |T| \, \|y\|_2^2 \le s \|y\|_2^2$ for every $y \in \R^T$, so
    \[
        \lambda \|y\|_2^2 + (1 - \lambda)\|y\|_1^2 \ge \left(\tfrac{\lambda}{s} + (1 - \lambda)\right)\|y\|_1^2,
    \]
    which yields $N^*_\lambda(g)^2 \le \left(\tfrac{\lambda}{s} + 1 - \lambda\right)^{-1} = \tfrac{s}{\lambda + (1 - \lambda) s}$.

    For the matching lower bound, choose any $T \subseteq [d]$ with $|T| = s$ and take $g = x = \mathbb{1}_T$, the indicator vector of $T$; then $g \in \mathcal{G}$.
    Since $\langle g, x\rangle = s$, $\|x\|_2^2 = s$, and $\|x\|_1^2 = s^2$,
    \[
        N^*_\lambda(g)^2 \ge \frac{\langle g, x\rangle^2}{\lambda \|x\|_2^2 + (1 - \lambda)\|x\|_1^2} = \frac{s^2}{\lambda s + (1 - \lambda) s^2} = \frac{s}{\lambda + (1 - \lambda) s}. \qedhere
    \]
\end{proof}

\subsection{Gradient and diameter bounds for block norms}\label{sec: gradient-and-diameter-bounds}

Consider an \gls{oco} instance over feasible set $\convexbody \subseteq \R^d$ with $s$-sparse losses \cref{assmp: loss-functions}, so that each loss function has a fixed containing support of size at most $s$ and gradients bounded in $[-1,1]^d$.
The bound uses Bernstein's concentration inequality (\cref{lem: bernstein}) for negatively associated random variables.

\begin{restatable}[Sparse gradient bound]{lemma}{sparseVectorBlockVarianceBound}\label{lem: sparse-vector-block-variance-bound}
    For $1\le s,n\le d$, define the uniform root-mean-square dual norm bound over fixed sparse vectors by
    \[
        G_n(s):=\sup_{\substack{c\in[-1,1]^d\\|\operatorname{supp}(c)|\le s}}
        \left(\E_{\mathcal B}\left[\|c\|_{[n]}^{*2}\right]\right)^{1/2},
    \]
    where $\mathcal B$ is a uniformly random equal-size $n$-partition of $[d]$.
    Then
    \[
        G_n(s)\le 3\max\left\{\sqrt{s/n},\sqrt{\ln n}\right\}.
    \]
\end{restatable}

\begin{proof}
    Fix $c$ and a size-$s$ set $C$ containing its support.
    Write $N_j=|C\cap B_j|$.
    Since
    \[
        \|c\|_{[n]}^{*2}\le \max_j N_j,
    \]
    it suffices to bound the expected maximum occupancy.
    For $n=1$, this is at most $s$.
    Suppose $n\ge2$, and condition on the block sizes $b_j=|B_j|$, which differ by at most one.
    Put $M=\max\{s/n,\ln n\}$.
    For each fixed $j$, the indicators $Y_i=\mathbf1\{i\in B_j\}$, $i\in C$, are negatively associated, since $B_j$ is a uniform size-$b_j$ subset of $[d]$ \cite{joag1983negative}.
    Their centered versions satisfy
    \[
        |Y_i-b_j/d|\le1,\qquad
        \mu_j:=\E[N_j]=sb_j/d\le2s/n,\qquad
        \sum_{i\in C}\E[(Y_i-b_j/d)^2]\le\mu_j.
    \]
    Here $b_j\le\lceil d/n\rceil\le2d/n$.
    Applying \cref{lem: bernstein} with $\delta=4M$ gives
    \[
        \Pr(N_j\ge\mu_j+4M)
        \le\exp\left(-\frac{48M^2}{8M+6\mu_j}\right)
        \le e^{-12M/5}\le n^{-2}.
    \]
    A union bound and $\max_jN_j\le s$ therefore yield
    \[
        \E[\max_jN_j]\le2s/n+4M+s/n\le7M.
    \]
    These bounds hold for every allowed choice of block sizes, so they also hold without conditioning.
    Taking square roots and using $\sqrt7 < 3$ proves the claim uniformly over $c$.
    \qedhere
\end{proof}

Under \cref{assmp: loss-functions}, the bound also controls $G_n$ despite the partition dependence of the iterates: for each round $t$,
\[
    \sup_{x\in\convexbody}\|\nabla f^{(t)}(x)\|_{[n]}^{*2}
    \le \max_{B\in\mathcal B}|B\cap C_{f^{(t)}}|
    =\| \mathbf1_{C_{f^{(t)}}}\|_{[n]}^{*2}.
\]
The vector on the right is fixed independently of the partition, so taking expectations and then the maximum over $t$ gives $G_n\le G_n(s)$.

Next, we relate the root-mean-square diameter $D_n$ to the block norm radius of $\convexbody$.

\begin{restatable}[Diameter bound]{lemma}{blockNormDiameter}\label{lem: block-norm-diameter-bound-general}
    Consider a nonempty convex set $\convexbody$ with block norm radius $\rho:=\rho_n(\convexbody)$ as defined in \cref{eqn: block-norm-radius}, and a fixed starting point $x^{(1)}\in\convexbody$.
    Then, the root-mean-square diameter $D_n$ defined in \cref{eqn: defn-block-norm-diameter} satisfies
    \[
        D_n = O(\rho \sqrt{\ln n}).
    \]

\end{restatable}

\begin{proof}
    Denote the sampled equal-size blocks as $\mathcal{B} = \{B_1, \dots, B_n\}$ and the corresponding mirror map for radius $\rho > 0$ from \cref{cor: scaled-mirror-maps} as $h := h_{\mathcal B,\rho}$.
    Then, $D_n = \sqrt{\E_{\mathcal B} \max_{z \in \convexbody} B_h(z \| \var{x}{1})}$.

    Abbreviate the parameters $p=p_n$ and $\gamma=\gamma_n$ (defined in \cref{eqn: block-norm-strongly-convex-function}).
    Denote $x := \var{x}{1}$ for brevity.
    Recall that $\rho = \rho_n(\convexbody)$ is an upper bound on any equal-size block norm on $\convexbody$ \cref{eqn: block-norm-radius}.

    Then, from \cref{eqn: block-norm-strongly-convex-function} and \cref{cor: scaled-mirror-maps}, we get
    \[
        h(x) = \rho^2 \cdot \frac{1}{\gamma p}\sum_{j \in [n]}\| (1/\rho) x_{B_j}\|_2^p = \frac{\rho^{2 - p}}{\gamma p} \sum_{j \in [n]}\|x_{B_j}\|_2^p.
    \]
    It is easy to verify that $\langle\nabla h(x), x\rangle= p \cdot h(x)$.
    Therefore, denoting $x := \var{x}{1}$, we get
    \begin{align*}
        B_h(z \| x) &:= h(z) - h(x) - \langle \nabla h(x), z - x \rangle = h(z) + (p - 1) h(x) - \langle \nabla h(x), z \rangle.
    \end{align*}

    First, we bound $h(z) + (p - 1) h(x)$ by $\rho^2/\gamma$.
    Since $p \ge 1$ and $\|(1/\rho) x\|_{[\mathcal{B}]} \le 1$ for all $x \in \convexbody$, we get from the above that $h(x) = \rho^2 \cdot \frac{1}{\gamma p}\sum_{j \in [n]}\| (1/\rho) x_{B_j}\|_2^p \le \rho^2 \cdot \frac{1}{\gamma p}\sum_{j \in [n]}\| (1/\rho) x_{B_j}\|_2 = \tfrac{\rho}{\gamma p} \|x\|_{[\mathcal{B}]} \le \tfrac{\rho^2}{\gamma p}$.
    Therefore,
    \[
        h(z) + (p - 1) h(x) \le (1 + (p - 1)) \cdot \tfrac{\rho^2}{\gamma p} = \tfrac{\rho^2}{\gamma}.
    \]

    For the inner product term, applying the Cauchy-Schwarz inequality blockwise,
    \begin{align*}
        -\gamma \langle \nabla h(x), z \rangle
        &= - \rho^{2 - p} \sum_{j \in [n]} \|x_{B_j}\|_2^{p - 2}
        \langle x_{B_j}, z_{B_j}\rangle
        \\
        &\le \rho^{2 - p} \sum_{j \in [n]} \|x_{B_j}\|_2^{p - 1} \|z_{B_j}\|_2 = \rho^2 \sum_{j \in [n]} \|(1/\rho)x_{B_j}\|_2^{p - 1} \|(1/\rho)z_{B_j}\|_2.
    \end{align*}
    For any block $j$, we have $\|(1/\rho) x_{B_j}\|_2 \le \|(1/\rho) x\|_{[\mathcal{B}]} \le 1$.
    Since $p -1  \ge 0$, we get $\|(1/\rho)x_{B_j}\|_2^{p - 1} \le 1$.
    Therefore, $-\gamma \langle \nabla h(x), z \rangle \le \rho^2 \sum_{j \in [n]} \|(1/\rho)z_{B_j}\|_2 \le \rho^2$.

    Together, using $1/\gamma = O(\ln n)$, we get
    \[
        D_n = \sqrt{\E_{\mathcal B} \max_{z \in \convexbody} B_h(z \| \var{x}{1})} \le \sqrt{\tfrac{2}{\gamma} \rho^2} = O(\rho \sqrt{\ln n}). \quad \qedhere
    \]
\end{proof}

\subsection{Gains in regret using block norms}\label{sec: gains-using-block-norms}

We prove \cref{thm: upper-bound-improvements}, which shows up to $\poly(d)$ gains in regret for several families of convex bodies and corresponding sparsities, and \cref{cor: sandwich-gain}.

\cref{tab: gains} details these results.
For each convex set in the table except the probability simplex, we prove more general results for the larger class (e.g., all boxes, ellipsoids, $L_p$ norm balls etc), and then show that there is some member in the family that achieves $\poly(d)$ gain in regret using block norms.

\RegretUpperBoundImprovements*

Throughout this subsection, we compare uniform regret bounds over the sparse-loss class and abbreviate $G_n(s)$ by $G_n$.
All mirror maps use the radius $\rho_n(\convexbody)$ of their respective body.
The displayed logarithmic gain bounds are for $s\ge3$.

For each convex set $\convexbody \subseteq \R^d$, the proof strategy is as follows: we need to show that there is some $n \in [d]$ and some $s$ such that $D_n(\convexbody)G_n(s)\le\gamma^{-1}\min\{D_\euc G_\euc,D_\ent G_\ent\}$ for the target gain $\gamma$.
Recall that $D_\euc G_\euc = D_1 G_1$ and $D_\ent G_\ent = D_d G_d$.
Lower bounds for the Euclidean and entropic geometries are computed explicitly.

Upper bounds on $G_n$ follow from \cref{lem: sparse-vector-block-variance-bound}.
To upper bound $D_n$, we use \cref{lem: block-norm-diameter-bound-general} and then bound $\rho_n(\convexbody)$ uniformly over equal-size partitions by using the specific geometry of $\convexbody$.
We illustrate this for the probability simplex and for an axis-aligned ellipsoid next.
Proofs for other examples are deferred to \cref{app: block-norm-computations}.

\paragraph{Probability simplex $\simplex_d$.}
We show the following gain:
\begin{align*}
    \gain(\simplex_d, s) = \Omega\left(\tfrac{\min(\sqrt{s}, \sqrt{\ln d})}{\ln s}\right).
\end{align*}

In particular, for $s = \ln d$, $\gain(\simplex_d, s) = \Omega\left(\tfrac{\sqrt{\log d}}{\log \log d}\right)$.

To see this, note that for the starting point $\var{x}{1} = (1/d)\mathbf{1}_d$, the diameter bounds $D_d = D_\ent = \Theta(\sqrt{\ln d})$ and $D_1 = D_\euc = \Theta(1)$ follow by evaluating the corresponding divergences at the vertices:
\[
    D_1^2=\tfrac12\left(1-\frac1d\right),\qquad
    D_d^2=\frac{1-d^{1-p_d}}{\gamma_dp_d}
    =\frac{e-1}{p_d}\ln d.
\]
Indeed, each divergence is convex in its first argument, and the gradient of the mirror map at $(1/d)\mathbf1_d$ is constant across coordinates, so the linear term vanishes on the simplex.
For arbitrary $n \in [d]$, using \cref{lem: block-norm-diameter-bound-general}, we get $D_n=O(\rho_n(\simplex_d)\sqrt{\ln n})$.
Since $\|\cdot\|$ is convex, the maximum occurs at a vertex, and therefore, for every equal-size partition $\mathcal B$, $\max_{z\in\simplex_d}\|z\|_{[\mathcal B]}=\max_{i\in[d]}\|\mathbf e_i\|_{[\mathcal B]}=1$.
Hence $\rho_n(\simplex_d)=1$, so that $D_n = O(\sqrt{\ln n})$.

Consequently, choosing $n = s$, and using \cref{lem: sparse-vector-block-variance-bound}, we get
\[
    \gain(\simplex_d, s) \ge \tfrac{\min(D_1 G_1(s), D_d G_d(s))}{D_s G_s(s)} = \Omega\left(\tfrac{\min(\sqrt{s}, \sqrt{\ln d})}{\sqrt{\ln s} \cdot \sqrt{\ln s}}\right).
\]

\paragraph{Axis-aligned ellipsoids.}
Let $w_1 \ge \dots \ge w_d > 0$ and consider the ellipsoid $E(w) := \{x: \sum_i x_i^2/w_i^2 \le 1\}$, with $\|w\|_2 \le 1$.
We show the following gain using block norms:
\begin{equation}\label{eqn: gain-ellipsoid}
    \gain(E(w), s) = \Omega\left(\tfrac{\min(w_1 \sqrt{s}, \|w\|_2)}{\|w_{\le s}\|_2 \log s}\right).
\end{equation}

\emph{Spiked ellipsoid.}
In particular, take $w_1=2^{-1/2}d^{-1/4}$ and $w_j=2^{-1/2}d^{-1/2}$ for $j\ge2$ (a spiked covariance), and choose $s = d^{1/2}$.
Then $\|w\|_2 \ge \sqrt{(d-1)/(2d)}\ge\tfrac12$, and $\|{w_{\le s}}\|_2 \le\bigl(\tfrac12d^{-1/2}+2\sqrt d\cdot\tfrac1{2d}\bigr)^{1/2} =\sqrt{\tfrac32}\,d^{-1/4}$.
Therefore,
\[
    \gain(E(w), s) = \Omega\left(\tfrac{d^{1/4}}{\ln d}\right).
\]

We prove the bound on gain next.
First, note that the containment $E(w)\subseteq\{x:\|x\|_1\le1\}$ follows from $\|w\|_2 \le 1$ since for any $x \in E(w)$, we have $\|x\|_1 = \sum_{i} w_i \cdot \tfrac{|x_i|}{w_i} \le \|w\|_2 \sqrt{\sum_i \tfrac{x_i^2}{w_i^2}} \le \|w\|_2$.

This common normalization is without loss of generality, but each mirror map still uses its own radius $\rho_n(E(w))$.
Indeed, for $c>0$, $\rho_n(cA)=c\rho_n(A)$ and
\[
    B_{h_{\mathcal B,\rho_n(cA)}}(cz\|cx^{(1)})
    =c^2B_{h_{\mathcal B,\rho_n(A)}}(z\|x^{(1)}).
\]
Thus all diameters scale by $c$ when the body and initialization are scaled together, while the uniform factors $G_n(s)$ are unchanged; the common factor cancels in the gain.

The next proposition provides bounds on diameters.

\begin{proposition}[Diameter bounds for axis-aligned ellipsoids]\label{prop: ellipsoid}
    Given the axis-aligned ellipsoid $E(w) := \{x: \sum_i x_i^2/w_i^2 \le 1\}$  with $w_1 \ge \dots \ge w_d > 0$ and $\|w\|_2 \le 1$, and a starting point $x^{(1)}\in E(w)$,
    \begin{itemize}
        \item $D_{\ent}\ge\tfrac12\|w\|_2$ and $D_{\euc} \ge \tfrac{1}{2} w_1$.
        \item $D_n = O(\sqrt{\ln n}) \cdot \|w_{\le n}\|_2$, where $\|w_{\le n}\|_2 := \left(\sum_{j \in [n]} w_j^2\right)^{1/2}$
    \end{itemize}
\end{proposition}

\begin{proof}
    For the lower bounds on $D_\ent, D_\euc$, we use strong convexity of the normalized endpoint mirror maps and symmetry of the ellipsoid.
    For either endpoint norm $\|\cdot\|$, write $D_h$ for the corresponding diameter and $a=x^{(1)}$.
    For $z,-z\in E(w)$, strong convexity and the triangle inequality give
    \begin{align*}
        2D_h^2&\ge B_h(z\|a)+B_h(-z\|a)\\
        &\ge\frac12\bigl(\|z-a\|^2+\|-z-a\|^2\bigr)\\
        &\ge\frac14\bigl(\|z-a\|+\|-z-a\|\bigr)^2
        \ge\|z\|^2.
    \end{align*}
    Thus $D_h\ge\|z\|/\sqrt2\ge\|z\|/2$.
    Since $w_1 e_1 \in E(w)$, we get $2D_\euc\ge\|w_1e_1\|_2=w_1$.
    Similarly, since $\tfrac{1}{\|w\|_2}\left(w_1^2, \dots, w_d^2\right) \in E(w)$, we get $2D_\ent\ge\left\|\tfrac1{\|w\|_2}(w_1^2,\dots,w_d^2)\right\|_1=\tfrac{\|w\|_2^2}{\|w\|_2}=\|w\|_2$.

    For the upper bound on $D_n$, we use \cref{lem: block-norm-diameter-bound-general} with $\rho_n(E(w))$.
    For any $n$ block partition $\mathcal{B} = (B_1, \dots, B_n)$, we first claim that the maximum block norm in $\convexbody$ is at most $\sqrt{\sum_{j: B_j \ne\emptyset}M_j^2}$, where we denote $M_j := \max_{i \in B_j} w_i$ for each block $j = 1, \dots, n$.
    The upper bound follows from this claim, since for any $n$ blocks (and not just random blocks), $M_1^2 + \dots + M_n^2 \le w_1^2 + \dots + w_n^2$ since $w_1, \dots, w_n$ are the highest $n$ values among $\{w_i: i \in [d]\}$.
    Since the claim holds for every partition, it implies $\rho_n(E(w))\le\|w_{\le n}\|_2$, and hence $D_n=O(\sqrt{\ln n})\|w_{\le n}\|_2$.

    To see this claim, fix $x \in \convexbody$, and write $r_j^2 := \sum_{i \in B_j} x_i^2/w_i^2$.
    Within block $B_j$, the corresponding $L_2$ norm is $\|x_{B_j}\|_2^2 = \sum_{i \in B_j} x_i^2 = \sum_{i \in B_j} w_i^2 (x_i^2/w_i^2) \le M_j^2 r_j^2$.
    Therefore, the block norm of $x$ is at most $\sum_{j = 1}^n \|x_{B_j}\|_2 \le \sum_{j \in [n]} M_j r_j \le \sqrt{\sum_{j \in [n]} M_j^2} \cdot \sqrt{\sum_{j \in [n]} r_j^2}$.
    However, by definition of the ellipsoid, $1 \ge \sum_{i \in [d]} (x_i^2/w_i^2) = \sum_{j \in [n]} r_j^2$.
\end{proof}

This gives the desired bound \cref{eqn: gain-ellipsoid} on the gain, by noting that $G_\euc = \sqrt{s}$, $G_\ent = 1$, and by choosing $n = s$, so that $G_s = O(\sqrt{\log s})$ from \cref{lem: sparse-vector-block-variance-bound}.

\subsubsection{Gains for sandwiched convex sets}

We prove \cref{cor: sandwich-gain}.
First, we prove the following lemma:
\begin{restatable}{lemma}{SandwichLemma}\label{lem: sandwich-lemma}
    Given nonempty bounded convex sets $\convexbody_0 \subseteq  \convexbody$ and a common starting point $\var{x}{1}\in\convexbody_0$, fixed independently of the random partitions, the corresponding root-mean-square diameters \cref{eqn: defn-block-norm-diameter} for the $n$th random block norm satisfy the following:
    \begin{enumerate}
        \item $D_n(\convexbody_0) \le D_n(\convexbody)$.
        \item If $\convexbody_0$ is centrally symmetric about the origin ($\convexbody_0=-\convexbody_0$), then $D_n(\beta \convexbody_0) \le \beta D_n(\convexbody_0)$ for all $\beta \ge 1$, with the starting point unchanged.
    \end{enumerate}
\end{restatable}

\begin{proof}
    For a convex set $A$, let $\rho(A):=\rho_n(A)$ be its block norm radius.
    With the common starting point $x:=\var{x}{1}$, the root-mean-square diameter is
    \[
        D_n(A)^2=\E_{\mathcal B}\left[\sup_{z\in A}
        B_{n,\rho(A)}(z\|x)\right].
    \]
    Fix sampled blocks $\mathcal B=\{B_1,\dots,B_n\}$ and abbreviate $p=p_n$ and $h=h_{\mathcal B}$.
    Homogeneity gives
    \[
        h_{\mathcal B,\rho}(x)=\rho^2h(x/\rho)=\rho^{2-p}h(x),
        \qquad
        \nabla h_{\mathcal B,\rho}(x)=\rho\nabla h(x/\rho)
        =\rho^{2-p}\nabla h(x).
    \]
    Consequently, $B_{n,\rho}(z\|x)=\rho^{2-p}B_h(z\|x)$.
    Since $\rho(\convexbody_0)\le\rho(\convexbody)$, $2-p\ge0$, and Bregman divergences are nonnegative,
    \begin{align*}
        \sup_{z\in\convexbody_0}B_{n,\rho(\convexbody_0)}(z\|x)
        &=\rho(\convexbody_0)^{2-p}\sup_{z\in\convexbody_0}B_h(z\|x)\\
        &\le\rho(\convexbody)^{2-p}\sup_{z\in\convexbody_0}B_h(z\|x)\\
        &\le\rho(\convexbody)^{2-p}\sup_{z\in\convexbody}B_h(z\|x)
        =\sup_{z\in\convexbody}B_{n,\rho(\convexbody)}(z\|x).
    \end{align*}
    Taking expectations and square roots proves part (1).

    For part (2), suppose $\convexbody_0=-\convexbody_0$.
    We first show that moving the starting point towards the origin cannot increase the supremum of the divergence over $\convexbody_0$.
    Since $h$ is even and homogeneous of degree $p$, we have $\nabla h(tx)=t^{p-1}\nabla h(x)$ and $\langle\nabla h(x),x\rangle=ph(x)$.
    Hence, for $0\le t\le1$,
    \begin{align*}
        \sup_{z\in\convexbody_0}B_h(z\|tx)
        &=(p-1)t^ph(x)
        +\sup_{z\in\convexbody_0}
        \bigl[h(z)-t^{p-1}\langle\nabla h(x),z\rangle\bigr]\\
        &=(p-1)t^ph(x)
        +\sup_{z\in\convexbody_0}
        \bigl[h(z)+t^{p-1}|\langle\nabla h(x),z\rangle|\bigr].
    \end{align*}
    The last equality follows by pairing $z$ with $-z$.
    The right-hand side is nondecreasing in $t$, since $h(x)\ge0$ and $p>1$.
    In particular,
    \[
        \sup_{z\in\convexbody_0}B_h(z\|x/\beta)
        \le\sup_{z\in\convexbody_0}B_h(z\|x).
    \]
    Finally, $\rho(\beta\convexbody_0)=\beta\rho(\convexbody_0)$ and homogeneity of the divergence give
    \begin{align*}
        \sup_{y\in\beta\convexbody_0}B_{n,\rho(\beta\convexbody_0)}(y\|x)
        &=(\beta\rho(\convexbody_0))^{2-p}
        \sup_{z\in\convexbody_0}B_h(\beta z\|x)\\
        &=\beta^2\rho(\convexbody_0)^{2-p}
        \sup_{z\in\convexbody_0}B_h(z\|x/\beta)\\
        &\le\beta^2\sup_{z\in\convexbody_0}
        B_{n,\rho(\convexbody_0)}(z\|x).
    \end{align*}
    Taking expectations and square roots proves part (2).
    \qedhere
\end{proof}

\SandwichGain*

\begin{proof}
    Apply part (1) of \cref{lem: sandwich-lemma} to both inclusions $\convexbody_0\subseteq\convexbody\subseteq\beta\convexbody_0$, and then apply part (2).
    For every $n\in[d]$, with the same starting point in all three bodies,
    \[
        D_n(\convexbody_0)\le D_n(\convexbody)
        \le D_n(\beta\convexbody_0)\le\beta D_n(\convexbody_0).
    \]
    The uniform gradient factors $G_n(s)$ are the same for both bodies.
    Thus
    \begin{align*}
        \gain(\convexbody,s)
        &=\frac{\min_{n\in\{1,d\}}D_n(\convexbody)G_n(s)}
        {\min_{n\in[d]}D_n(\convexbody)G_n(s)}\\
        &\ge\frac{\min_{n\in\{1,d\}}D_n(\convexbody_0)G_n(s)}
        {\beta\min_{n\in[d]}D_n(\convexbody_0)G_n(s)}
        =\frac1\beta\,\gain(\convexbody_0,s). \qedhere
    \end{align*}
\end{proof}

\subsection{Instances with realized gains}\label{sec: instances-with-realized-gains}

The gains in \cref{tab: gains} and \cref{thm: upper-bound-improvements} measure improvements in regret \emph{upper bound} over Euclidean and entropic geometries.
In this section, we prove \cref{thm: regret-improvement-using-block-norms} and give explicit \gls{oco} instances for two polytopes where large gains in regret are realized using block norms.

\regretImprovementsBlockNorm*

In both cases, given dimension $d > 0$ and time horizon $T > 0$, the sequence of loss functions is generated as follows: at each time step $t$, the loss function is defined as $f^{(t)}(x) := - \langle c^{(t)},  x\rangle$ where $c^{(t)}$ is a $0$-$1$ vector with exactly $s$ non-zero coordinates, chosen as follows: $c^{(t)}_1 = 1$ and the other $s - 1$ non-zero coordinates are chosen uniformly at random from $[2, d]$.
The sparsity $s$ and starting point $\var{x}{1}$ are chosen separately for each example.

We follow a two-fold proof outline.
The regret for the $s$th block norm is upper bounded through \cref{eqn: block-norm-omd-regret-bound}.
To lower bound the regret for \gls{opgd}/\gls{oeg}, we will show that either algorithm already incurs large enough regret by some threshold time $T_0\le T$.
In both constructions, $x^*=\mathbf e_1$ minimizes every loss function over the polytope, so regret against $x^*$ is nonnegative at every time step, and therefore time steps beyond $T_0$ can be omitted.

\subsubsection{Logarithmic regret improvement for simplex}

Assume $d$ is sufficiently large and $T\ge C\ln d$ for a sufficiently large absolute constant $C>0$.
Choose sparsity $s = \lfloor\ln d\rfloor$ and starting point $\var{x}{1} = \tfrac{1}{d}\mathbf{1}_d$.
As in \cref{sec: gains-using-block-norms}, choosing $n = s = \lfloor\ln d\rfloor$ blocks and using \cref{lem: sparse-vector-block-variance-bound} and \cref{lem: block-norm-diameter-bound-general} yields the regret upper bound $\E[\regret_s(T)] = O((\ln s)\sqrt{T}) = O((\ln \ln d) \sqrt{T})$.

\paragraph{Regret lower bounds.}

Since $c_1 = 1$ always, $\E[\var{f}{t}(x^*)] = - \E[c^\top x^*]= - c_1 x^*_1 = -x^*_1$.
Further, for any $x \in \simplex_d$,
\begin{align*}
    & \E[f^{(t)}(x)] = -\sum_{i \in [d]} \E[c^{(t)}_i] x_i = -\sum_{i \in [d]} \Pr(c_i^{(t)} = 1) x_i \\
    = & - x_1 - \tfrac{s - 1}{d - 1} \sum_{i \in [2, d]} x_i = -x_1 - \tfrac{s - 1}{d - 1}(1 - x_1) \ge - x_1 - \tfrac{s - 1}{d - 1}. \notag
\end{align*}
Therefore, $\E[f^{(t)}(x) - \var{f}{t}(x^*)] \ge (x^*_1 - x_1) - \tfrac{s - 1}{d - 1}$.

Suppose the iterates of an algorithm $\var{x}{t}$ satisfy $1 - \var{x}{t}_1 = x^*_1 - \var{x}{t}_1 \ge \gamma$ for all $t \le T_0$, for some  $T_0 \in [T]$ and some $\gamma > 0$.
Then, the above gives
\begin{equation}\label{eqn: logarithmic-improvement-coordinate-regret-bound}
    \sum_{t \in [T_0]} \E[\var{f}{t}(\var{x}{t}) - \var{f}{t}(x^*)] \ge \left(\gamma - \tfrac{s - 1}{d - 1}\right)T_0.
\end{equation}

\emph{\gls{opgd}.}
We claim that for standard step size $\eta = \tfrac{D_\euc}{G_\euc \sqrt{T}}$ that $x^*_1 - \var{x}{t}_1 \ge \gamma := 1/2 - 1/d$ for all $t \le T_0 := 1 + \left\lfloor\frac{1}{2\eta}\right\rfloor$.
Plugging this into \cref{eqn: logarithmic-improvement-coordinate-regret-bound} and noting that $\gamma - \tfrac{s - 1}{d - 1} \ge \tfrac{1}{4}$ for all large enough $d$ and $s = \lfloor\ln d\rfloor$ gives
\[
    \E[\regret_\euc(T)] \ge \E[\regret_\euc(T_0)] \ge \left(\gamma - \tfrac{s - 1}{d - 1}\right)T_0 = \Omega\left(\sqrt{sT}\right) = \Omega(\sqrt{T \ln d}).
\]

To prove the claim, recall that $h_\euc(x)=\tfrac12\|x\|_2^2$, so the Bregman diameter is
\[
    D_\euc = \sqrt{\tfrac12\max_{z\in\simplex_d}\|z-x^{(1)}\|_2^2} = \sqrt{\tfrac12(1-1/d)}.
\]
Since $G_\euc=\sqrt s$, the step size is $\eta=D_\euc/(G_\euc\sqrt T)=\sqrt{(1-1/d)/(2sT)}$.
The Euclidean projection onto the simplex subtracts a nonnegative threshold after adding $\eta c^{(t)}$, so $x_1^{(t+1)}\le x_1^{(t)}+\eta$.
By induction, $x_1^{(t)}\le 1/d+(t-1)\eta$.
Thus, for $t\le T_0=1+\lfloor1/(2\eta)\rfloor\le T$, we have $x_1^{(t)}\le 1/d+1/2$ and hence $1-x_1^{(t)}\ge1/2-1/d$.

\emph{\gls{oeg}.}
We claim that for standard step size $\eta = \tfrac{D_\ent}{G_\ent \sqrt{T}}$ that $x^*_1 - \var{x}{t}_1 \ge \gamma := 1/2$ for all $t \le T_0 := \left\lfloor(\sqrt{T \ln d})/2\right\rfloor$.
As for \gls{opgd}, plugging into \cref{eqn: logarithmic-improvement-coordinate-regret-bound} gives regret $\ge \tfrac{1}{4} T_0 = \Omega(\sqrt{T \ln d})$.

Recall that $h_\ent(x) := \sum_{i \in [d]} x_i \ln x_i$.
Standard computations show that for starting point $\var{x}{1} = \tfrac{1}{d} \mathbf{1}_d$, diameter $D_\ent = \sqrt{\max_{x \in \simplex_d}KL(x\|x^{(1)})} = \sqrt{\ln d}$ and gradient bound $G_\ent = \max_{t} \|- \var{c}{t}\|_\infty = 1$.
Therefore, the step size $\eta = \sqrt{\tfrac{\ln d}{T}}$.

Since $\nabla f^{(t)}(x) = -c^{(t)}$ for all $x$, we have that
\[
    x^{(t + 1)}_i =
    \begin{cases}
        \frac{x^{(t)}_i e^\eta}{e^{\eta} \sum_{j:c_j^{(t)}=1} x_j^{(t)} + \sum_{j:c_j^{(t)}=0} x_j^{(t)}} & \text{if} \ c^{(t)}_i = 1, \\
        \frac{x^{(t)}_i}{e^{\eta} \sum_{j:c_j^{(t)}=1} x_j^{(t)} + \sum_{j:c_j^{(t)}=0} x_j^{(t)}} & \text{if} \ c^{(t)}_i = 0.
    \end{cases}
\]
By induction on $t$, we get $x^{(t)}_1 \le \frac{1}{d} e^{\eta t}$.
In particular, when $t \le T_0 = \lfloor(1/2)\sqrt{T \ln d}\rfloor \le (\ln d)/(2 \eta)$, we have $x^{(t)}_1 \le (1/d) e^{\eta T_0} \le \tfrac{1}{\sqrt{d}} \le \tfrac{1}{2} = \gamma$ for all $d \ge 4$.

\subsubsection{Polynomial improvement in regret}\label{sec: polynomial-regret-improvement}

We prove Part (1) of \cref{thm: regret-improvement-using-block-norms}, with the polytope $P = \conv(\simplex_d \cup \{A \mathbf{1}_d\})$ for $A = d^{-3/4}$.
As before, the loss function at each time $t \in [T]$ is $\var{f}{t}(x) := - x^\top \var{c}{t}$, where $\var{c}{t} \in \{0, 1\}^d$ is a random binary vector with exactly $s$ non-zero coordinates, with $\var{c}{t}_1 = 1$ always, and remaining $s - 1$ non-zero coordinates chosen uniformly at random from the remaining $d - 1$ coordinates.
Let $s := \lfloor\sqrt d\rfloor$ and starting point $\var{x}{1} := A \mathbf{1}_d$.
Denote $R:=Ad=d^{1/4}$, and assume throughout that $d$ is sufficiently large and $T\ge C\sqrt d$ for a sufficiently large absolute constant $C$.

Since losses $\var{f}{t}$ are linear, the optimal $x^*$ is one of the vertices.
We claim $x^* = \mathbf{e}_1$: indeed, $\E[\var{f}{t}(\mathbf{e}_1)] = -1$, while $\E[\var{f}{t}(\mathbf{e}_i)] = -\tfrac{s - 1}{d - 1} > -1$ for $i \neq 1$, and $\E[\var{f}{t}(\var{x}{1})] = - A \sum_{i = 1}^d \E[\var{c}{t}_i] = - As \ge -d^{-1/4}>-1$.
In fact, $\var{f}{t}(\mathbf e_1)=-1\le\var{f}{t}(x)$ for every $x\in P$ and every realization of $\var{c}{t}$, since $As<1$.

We upper bound the regret of the $s$th block norm by $O(\sqrt T\ln d)$ using \cref{lem: sparse-vector-block-variance-bound} and \cref{lem: block-norm-diameter-bound-general}; this is deferred to \cref{sec: polynomial-regret-improvement-proofs}.
To lower bound regret for Euclidean $(n = 1)$ or entropic $(n = d)$ geometries, we follow the same strategy as before, and show that the algorithms have already accumulated large enough regret by some threshold time $T_0=\Theta(d^{1/4}\sqrt T)$, by showing that the iterates $\var{x}{t}, t \le T_0$ satisfy $x^*_1 - \var{x}{t}_1 \ge \gamma$ for appropriate $\gamma > 0$.

Indeed, for any $x \in P$ and any time step $t$,
\begin{equation*}
    \E [f^{(t)}(x)] = - \sum_{i \in [d]} \Pr(c^{(t)}_i = 1) x_i = - x_1 - \tfrac{s - 1}{d - 1} \sum_{i \in [2, d]} x_i \ge - x_1 - \tfrac{s}{d} \sum_{i \in [2, d]} x_i.
\end{equation*}
In particular, since $\sum_{i \in [2, d]} x_i \le \|x\|_1 \le \|A \mathbf{1}_d\|_1 = Ad$ for all $x \in P$, we can lower bound $\E [f^{(t)}(x)] \ge -x_1 - \frac{s}{d}(Ad) = -x_1 - As \ge -x_1 - d^{-1/4}$.
Therefore, since $\E[\var{f}{t}(x^*)] = -1 = - x^*_1$, conditioning on the past and using independence of the current loss gives, whenever $x^*_1-\var{x}{t}_1\ge\gamma$ for $t\le T_0$,
\begin{equation}\label{eqn: polynomial-regret-pointwise-lower-bound}
    \sum_{t \le T_0} \E[\var{f}{t}(\var{x}{t}) - \var{f}{t}(x^*)] \ge \sum_{t \le T_0} \E\left[(x^*_1 - \var{x}{t}_1) - d^{-1/4}\right] \ge (\gamma - d^{-1/4}) T_0.
\end{equation}

\paragraph{Lower bound on \gls{opgd}.}
First, we compute the step size.
By symmetry, the diameter
\begin{align*}
    D_\euc^2 &= \frac12\max_{x \in P} \|x - x^{(1)}\|_2^2 = \frac12\|\mathbf{e}_1 - A\mathbf{1}_d\|_2^2 \\
    &= \frac12\big((1 - A)^2 + (d - 1)A^2\big)=\Theta(1).
\end{align*}
The last equality holds since $A = d^{-3/4}$.
The gradient bound is $G_\euc = \max_{t \in [T], x \in P} \|\nabla f^{(t)}(x)\|_2 = \sqrt{s} = \Theta(d^{1/4})$.
Therefore, $\eta = \frac{D_\euc}{G_\euc\sqrt{T}} = \Theta(1/(d^{1/4}\sqrt T))$.

At time step $t$, denote the set of non-zero coordinates in $\nabla f^{(t)} = - c^{(t)}$ as $C^{(t)} = \{i \in [d]: c^{(t)}_i = 1\}$.
The unconstrained update rule of \gls{opgd} for coordinates $i \in C^{(t)}$ is then $y^{(t)}_i = x^{(t)}_i + \eta$.
The Euclidean projection satisfies $x_1^{(t+1)}\le\max\{A,x_1^{(t)}+\eta\}$, as verified in \cref{app: polynomial-regret-euclidean-coordinate-bound}.
Therefore, we get by induction on $t$ that for the sequence $x^{(t)}$ of iterates of \gls{opgd}, $x_{1}^{(t)} \le A + (t - 1) \eta$.
In particular, for $T_0 := 1 + \lfloor 1/(2\eta)\rfloor = \Theta(d^{1/4}\sqrt T)\le T$, we get for $t \le T_0$ that $x^{(t)}_1 \le A + \frac{1}{2}$.
Choosing $\gamma := \frac12-A$ and plugging in \cref{eqn: polynomial-regret-pointwise-lower-bound},
\begin{align*}
    \E[\regret_\euc(T)] &\ge \sum_{t \in [T_0]} \E[f^{(t)}(x^{(t)}) - f^{(t)}(x^*)] \\
    &\ge (\tfrac12-d^{-3/4}-d^{-1/4}) T_0 = \Omega(T_0) = \Omega(d^{1/4} \sqrt{T}).
\end{align*}

\paragraph{Lower bound on OMD$_d$.}
Since $P \not\subseteq \simplex_d$, we cannot directly use \gls{oeg}, and instead use OMD$_d$ with $d$ blocks, with the normalized mirror map $h_{d,R}$, since $\rho_d(P)=R$.

For iterates $\var{x}{t}$ of OMD$_d$ with the standard step size $\eta := \tfrac{D_\ent}{G_\ent \sqrt{T}}$, we prove the following:
\begin{restatable}{lemma}{polynomialRegretBoundEntropicCoordinateBound}\label{lem: polynomial-lower-bound-entropic-coordinate-bound}
    The iterates $x^{(1)}, \ldots, x^{(T)} \in P$ of OMD with $d$th block norm satisfy, for every realization of the loss sequence,
    \[
        x_1^{(t)} \le A + \frac{\eta}{(Ad)^2}(t - 1),
        \qquad \eta=\Theta\left(\frac{Ad}{\sqrt T}\right).
    \]
\end{restatable}

In particular, the lemma implies that for $T_0 := 1 + \lfloor R^2/(2\eta)\rfloor = \Theta(d^{1/4}\sqrt T)\le T$, the iterates $\var{x}{t}, t\le T_0$ satisfy $x^*_1 - \var{x}{t}_1 \ge \tfrac{1}{2} - A$.
Denoting $\gamma := \tfrac{1}{2} - A$ and noting that $A = d^{-3/4}$, \cref{eqn: polynomial-regret-pointwise-lower-bound} implies
\begin{align*}
    \E[\regret_d(T)] &\ge \sum_{t \in [T_0]} \E[f^{(t)}(x^{(t)}) - f^{(t)}(x^*)] \\
    &\ge (\tfrac12-d^{-3/4}-d^{-1/4}) T_0 = \Omega(T_0) = \Omega(d^{1/4} \sqrt{T}).
\end{align*}

Therefore, it remains to prove \cref{lem: polynomial-lower-bound-entropic-coordinate-bound} in order to prove \cref{thm: regret-improvement-using-block-norms}.

Since $\|A\mathbf1_d\|_1=Ad=R>1$, define the scaled polytope $\hat P=P/R$ to be the convex hull of $\mathbf e_1/R,\ldots,\mathbf e_d/R$ and $\mathbf1_d/d$.
The loss functions are scaled up by $R$, so $\hat f^{(t)}(z)=f^{(t)}(Rz)=-R\langle c^{(t)},z\rangle$.
Since
\[
    B_{d,R}(Rz\|Rw)=R^2 B_{h_d}(z\|w),
\]
OMD with $h_{d,R}$ on $P$ and step size $\eta$ is equivalent to OMD with $h_d$ on $\hat P$ and step size $\hat\eta=\eta/R^2$.
Let $\var{z}{t}\in\hat P$ be the scaled iterates, so that $\var{x}{t}=R\var{z}{t}$.
In \cref{sec: polynomial-regret-improvement-proofs}, we compute the diameter of $\hat P$ and use the projection optimality conditions to prove the coordinate bound directly.

\section{Learning the optimal mirror map}\label{sec: learning-optimal-mirror-map}

In this section, we discuss learning the optimal mirror map for a given \gls{oco} instance from among a family of mirror maps, e.g., block norm mirror maps.

\subsection{Alternating between mirror maps is suboptimal}\label{sec: alternating-mirror-maps-is-suboptimal}

The most natural way to combine mirror maps is to use different mirror maps at different time steps $t \in [T]$.
We construct instances where naively alternating between mirror maps can incur regret $\Omega(T)$, thus demonstrating that more sophisticated strategies are needed to combine mirror maps.
In particular, we give instances where using the \gls{opgd} and \gls{oeg} at alternate steps leads to $\Omega(T)$ regret.
Contrast this with \gls{opgd} or \gls{oeg}, each of which results in $O(\sqrt{T})$ regret.
Our result holds for \emph{all (fixed) step sizes}, thus showing that adjusting the step sizes is not sufficient to produce an optimal-regret algorithm.

\alternatingMirrorMaps*

For clarity, take $T$ to be a positive multiple of $32$.
To prove this result, we give a specific \gls{oco} instance in $2$ dimensions where the alternating algorithm incurs $\Omega(T)$ regret.

\paragraph{\gls{oco} instance.}
The convex body is the probability simplex $\simplex_2 = \{x \in \R^2: x_1 + x_2 = 1, x_1, x_2 \ge 0\}$ in $2$ dimensions.
All algorithms must start at $x^{(1)} = (1/2, 1/2)$.
We are given that the loss functions $f^{(t)}, t \in [T]$ all satisfy $\|\nabla f^{(t)}(x)\|_2 \le 2$ for all $x \in \R^2$.
These will be chosen before play begins. For simplicity, we present two separate sequences for `small' and `large' step sizes, but they can be made oblivious by sampling each of the two sequences with probability $1/2$ a priori, without the knowledge of step sizes.

Consider an algorithm that runs \texttt{MirrorDescentStep} (\cref{alg: mirror-descent-step}) with the Euclidean mirror map $h_\euc(x) = \frac{1}{2} \|x\|_2^2$ on odd steps and with the entropic mirror map $h_\ent = \sum_{i \in [d]} x_i \ln x_i$ on even steps.
Suppose the step sizes for these are $\eta_{\euc}$ and $\eta_{\ent}$ respectively.
Using a case analysis on the step sizes, we show that there always exist loss functions so that this algorithm has regret $\Omega(T)$.

\textbf{Case 1}: $\eta_{\euc} \ge \frac{16}{T}$.
The loss functions are specified as follows:
\begin{equation*}
    f^{(t)}(x) =
    \begin{cases}
        -x_1 & \text{if} \ t \ \text{is odd}, \\
        0 & \text{if} \ t \ \text{is even and} \ t \le T/8, \\
        -2 x_2 & \text{if} \ t \ \text{is even and} \ t > T/8.
    \end{cases}
\end{equation*}
Since $\sum_{t \in [T]}f^{(t)}(x) = - \frac{T}{2} x_1 - \frac{7T}{8} x_2$, the optimal point $x^* = \mathbf{e}_2 = (0, 1)$, and the optimal total loss is $-\frac{7T}{8}$.

We will claim that the algorithm hits and stays at vertex $\mathbf{e}_1 = (1, 0)$ in $\le \frac{T}{8}$ time steps and then does not move, i.e., $x^{(t)} = \mathbf{e}_1$ for all $t \ge \frac{T}{8}$.
Assuming this claim, the loss of the algorithm is at least $-\frac{1}{2} \cdot \frac{T}{8}$ in the first $\frac{T}{8}$ time steps, and exactly $- \frac{1}{2} \cdot \frac{7T}{8}$ thereafter.
Thus, the total loss of the algorithm is at least $-T/2$, and the regret is at least $-T/2+7T/8=3T/8=\Omega(T)$.

To see the claim, first note that $x^{(t + 1)} = x^{(t)}$ for all even $t \le T/8$ since $f^{(t)}(x) = 0$.
Therefore, any movement in $x^{(t)}$ is due to the Euclidean mirror map update step, which moves $x^{(t)}$ towards $\mathbf{e}_1$ since $\nabla f^{(t)}(x) = (-1, 0)$ for all odd $t$.
Specifically, for all odd $t$,
\[
    y^{(t + 1)}_1 = x^{(t)}_1 + \eta_{\euc}, \qquad x^{(t + 1)}_1 = \Pi_\euc(y^{(t + 1)}) = \min\left(x^{(t)}_1 + \frac{\eta_{\euc}}{2}, 1\right).
\]
Thus, as long as $x^{(t)}_1 < 1$, the Euclidean update step moves it towards $1$ by a distance of $\frac{\eta_{\euc}}{2}$ at every odd time step $t \le T/8$.
Since $\eta_{\euc} \ge \frac{16}{T}$, and since $x^{(1)}_1 = 1/2$, this implies that $x^{(T/8)}_1 = \min\{\frac{1}{2} + \frac{T}{16}\cdot\frac{\eta_\euc}{2},1\} = 1$.

Once $x^{(t)} = (1, 0)$, the entropic map cannot revive zero mass on $x^{(t)}_2$, since its update rule takes
\[
    x^{(t + 1)}_2 \propto x^{(t)}_2 \exp(- \eta_{\ent} \nabla_2 f^{(t)}(x^{(t)})) = 0.
\]

\textbf{Case 2}: $\eta_{\euc} < \frac{16}{T}$.
The loss functions are specified as follows:
\begin{equation*}
    f^{(t)}(x) =
    \begin{cases}
        -x_2 & \text{if} \ t \ \text{is odd}, \\
        0 & \text{if} \ t \ \text{is even}.
    \end{cases}
\end{equation*}
Clearly, the optimal point is $x^* = \mathbf{e}_2 = (0, 1)$, with optimal total loss $-T/2$.
As in Case 1, the Euclidean mirror map moves towards the optimal by a distance of $\le \frac{\eta_{\euc}}{2}$ at every odd time step.
There is no movement at even time steps.
That is, for all $t$,
\[
    x^{(t)}_2 \le \frac{1}{2} + \frac{\eta_\euc}{2} \cdot \frac{t}{2}.
\]
Therefore, since $\eta_{\euc} < \frac{16}{T}$, we have $x^{(t)}_2 \le \frac{3}{4}$ for all $t \le T/16$.
The regret up to $T/16$ is therefore $\ge \left(1 - \frac{3}{4}\right) \cdot \frac{T}{32} = \frac{T}{128}$.
Since the regret is nonnegative thereafter, the total regret is at least $\frac{T}{128}$.

\subsection{Hedge algorithm over mirror maps}\label{sec: mwu-over-mirror-maps}

Having established that intermediate block norms can yield significant gains, we now ask whether an algorithm can adaptively learn which mirror map to use -- without prior knowledge of the structure of loss functions.

We show that Hedge algorithm that maintains each mirror map as an `expert' obtains regret at most the best mirror map's regret, plus an additional $O(\rho \sqrt{\ln N} \sqrt{T})$ term, where $\rho \ge \max_{x, z \in \convexbody, t \in [T]} f^{(t)}(x) - f^{(t)}(z)$ is an upper bound on the loss function differentials.
The algorithm requires knowing $\rho$.
The following theorem formalizes this:

\combiningMirrorMaps*

For completeness, we include the algorithm and a proof of the theorem in \cref{app: mwu-proof}.
The proof follows standard arguments.

Next, we combine this idea with the $1 + \log_2 d$ block norm mirror maps corresponding to $n \in \{2^0, 2^1, \ldots, 2^{\log_2 d}\}$.
In this setting, even the input on the step sizes is not required: given the Euclidean diameter $D_\euc = \max_{x, z \in \convexbody} \|x - z\|_2$, and the Euclidean gradient norm bound $G_\euc = \max_{x \in \convexbody, t \in [T]} \|\nabla f^{(t)}(x)\|_2$, the diameter $D$ and gradient $G$ for any other block norm is within factor $d$ of $D_\euc$ and $G_\euc$ respectively.
Thus, we can search over the step sizes for each block norm (up to a factor of 2) by making $O(\log d)$ copies of each block norm mirror map, but with different step sizes\footnote{This search over step sizes shares the idea of aggregating algorithms with different learning rates with MetaGrad \cite{van2016metagrad}.}.
We have the following result:

\combiningBlockNorms*

\begin{proof}
    By the portfolio construction above and \cref{thm: mirror-weights}, taking expectations over the random partitions gives
    \[
        \E[\regret(T)] = O\left(\min_n D_nG_n\sqrt T + \rho\sqrt{T\ln\ln d}\right),
    \]
    where $\rho:=\max_{t\in[T]}\sup_{x,z\in\convexbody}(f^{(t)}(x)-f^{(t)}(z))$.
    To control $\rho$, fix $n$, a partition $\mathcal B$, and a round $t$, and write
    \[
        D_{\mathcal B}:=\left(\sup_{u\in\convexbody}B_{h_{\mathcal B,\rho_n(\convexbody)}}(u\|x^{(1)})\right)^{1/2},
        \qquad G_{\mathcal B,t}:=\sup_{u\in\convexbody}\|\nabla f^{(t)}(u)\|^*_{[\mathcal B]}.
    \]
    Strong convexity and the triangle inequality give $\|x-z\|_{[\mathcal B]}\le2\sqrt2 D_{\mathcal B}$ for every $x,z\in\convexbody$.
    Integrating the gradient along the segment from $z$ to $x$ therefore gives
    \[
        f^{(t)}(x)-f^{(t)}(z)\le G_{\mathcal B,t}\|x-z\|_{[\mathcal B]}\le2\sqrt2 D_{\mathcal B}G_{\mathcal B,t}.
    \]
    The loss sequence is independent of the partition, so taking expectations and applying Cauchy--Schwarz yields
    \[
        f^{(t)}(x)-f^{(t)}(z)\le2\sqrt2\left(\E[D_{\mathcal B}^2]\right)^{1/2}\left(\E[G_{\mathcal B,t}^2]\right)^{1/2}\le2\sqrt2 D_nG_n.
    \]
    Taking the supremum over $x,z$ and the maximum over $t$ gives $\rho\le2\sqrt2 D_nG_n$ for every $n$.
    Substituting $\rho\le2\sqrt2\min_nD_nG_n$ into the portfolio bound proves the claimed expectation bound.
\end{proof}

\section{Conclusion}\label{sec: mirror-maps-conclusion}

The choice of the mirror map in \gls{omd} captures the geometry of the \gls{oco} problem, and can be exploited to show regret improvements in specific regimes of loss functions.
As we showed, specific \emph{block norm mirror maps} can yield polynomial-in-dimension improvements in regret in certain regimes over standard algorithms like \gls{opgd} and \gls{oeg}.

However, when very little is known about the loss functions \emph{a priori}, adhering to a fixed mirror map may lead to suboptimal regret.
In this work, we show that a \emph{portfolio of block norm mirror maps} can provide significantly improved performance.
We show that, using the Hedge/multiplicative-weights algorithm, we can switch between various block norms at runtime, achieving the regret rate of the best block norm, up to an $O(\sqrt{\ln \ln d})$ factor.

A natural direction for future work is to explore other families of mirror maps that interpolate between Euclidean and entropic geometries and to investigate whether a small, universal portfolio of geometries can achieve near-optimal regret across all \gls{oco} instances.
In particular, this work focuses on block norms with equal-size blocks.
Incorporating unequal blocks (that sometimes better adapt to the loss functions with non-uniform sparsity patterns) may require new techniques since the number of such partitions is much larger (even if we allow an $O(\ln d)$-approximation).

\section*{Acknowledgements}

This work was supported by the AI Institute for Optimization (AI4OPT) and the National Science Foundation (NSF) under grants CCF-AF-2504994 and CCF-2106444.

\textbf{AI-use statement.} The large language model Claude Fable 5 was used to assist with presentation of the manuscript. All research ideas and proofs originated with the authors, who take full responsibility for the manuscript's content and the experiments.

\bibliographystyle{alpha}
\bibliography{references}

\newpage

\appendix

\section{Online mirror descent}\label{sec: mirror-maps-omd}

\begin{algorithm}
    \caption{\texttt{MirrorDescentStep}$(x, \mathcal{K}, h, \nabla f, \eta)$}\label{alg: mirror-descent-step}
    \KwIn{point $x$ in convex body $\mathcal{K}$, distance generating function $h$, loss function gradient $\nabla f$, and step size $\eta$}
    $\theta = \nabla h (x)$ \tcp*[r]{Map to dual space}
    $\theta' = \theta - \eta \nabla f(x)$ \tcp*[r]{Move in direction of negative gradient}
    $y = (\nabla h)^{-1}(\theta')$ \tcp*[r]{Map back to primal space}
    \Return $x' = \Pi^{(h)}_\mathcal{K}(y) := {\arg\min}_{z \in \mathcal{K}} B_h(z \| y)$ \tcp*[r]{Project on $\mathcal{K}$}
\end{algorithm}

\LinesNumbered

\begin{algorithm}
    \caption{\texttt{OnlineMirrorDescent} (\cite{shalev-shwartz_online_2012, hazan_introduction_2016})}
    \label{alg:mirror-descent}
    \KwIn{(i) an online convex optimization setup with a closed, bounded convex body $\mathcal{K}$, a time horizon $T$, (ii) a convex function $h$ defined\footnotemark on $\R^d$ with an invertible gradient $\nabla h$, and (iii) a starting point $x^{(1)} \in \mathcal{K}$}
    \let\oldnl\nl
    \newcommand{\nonl}{\renewcommand{\nl}{\let\nl\oldnl}}
    \nonl \textbf{Parameters}: step sizes $\eta^{(1)}, \ldots, \eta^{(T)} \in \R_{> 0}$ \;
    \For{$t = 1, \ldots, T$}{
        play $x^{(t)}$ and observe loss $f^{(t)}$\;
        $x^{(t + 1)} = \texttt{MirrorDescentStep}(x^{(t)}, \mathcal{K}, h, \nabla f^{(t)}, \eta^{(t)})$\;
    }
\end{algorithm}

\footnotetext{For simplicity, we assume that mirror map $h$ is defined on $\R^d$; however, one must be more careful and define it on an appropriate closed and convex set $X \supseteq \mathcal{K}$.}

We formally describe the mirror descent algorithm \cite{nemirovski1983problem, ben-tal_nemirovski_lecture_notes_2025, shalev-shwartz_online_2012, hazan_introduction_2016} and state its regret guarantees.

Suppose the mirror map $h: \convexbody \to \R$ is $1$-strongly convex with respect to some norm $\|\cdot\|$ on $\R^d$, i.e., if
\[
    B_h(x \| y) = h(x) - h(y) - \nabla h(y)^\top (x - y) \ge \frac{1}{2} \|x - y\|^2 \quad \ \forall \ x, y \in \convexbody.
\]

Denote by $\|\cdot\|^*$ the dual norm to $\|\cdot\|$.
Given starting point $x^{(1)}$ for \gls{omd}, denote by $D_h := \sqrt{\max_{z \in \mathcal{K}} B_h(z \| x^{(1)})}$ the diameter of $\mathcal{K}$ under $B_h$.
Then, for convex loss functions $f^{(1)}, \ldots, f^{(T)}$, the regret of \gls{omd} with Bregman divergence $B_h$, step-sizes $\eta_t = \eta$, and iterates $x^{(1)}, \ldots, x^{(T)}$ satisfies for any $x^* \in \mathcal{K}$:
\begin{align*}
    \regret(T) := \sum_{t \in [T]} \left( f^{(t)}(x^{(t)}) - f^{(t)}(x^*)\right) \le \frac{D_h^2}{\eta} + \frac{\eta}{2} \sum_{t \in [T]} \left(\|\nabla f^{(t)}(x^{(t)})\|^*\right)^2
\end{align*}
If we further have the bound $\|\nabla f^{(t)}(x^{(t)})\|^* \le G_h$, then this implies
\begin{align*}
    \regret(T) \le \frac{D_h^2}{\eta} + \frac{T \eta G_h^2}{2}.
\end{align*}
Choosing $\eta = \frac{D_h}{G_h\sqrt{T}}$ gives
\begin{align}\label{eqn: regret-bound-general}
    \regret(T) \le 2 D_hG_h \sqrt{T}
\end{align}
This also holds in expectation when gradients are \emph{stochastic} as long as the second moment is bounded:
\[
    \E \left[\left(\|\nabla f^{(t)}(x^{(t)})\|^*\right)^2\right] \le G_h^2.
\]

\section{Suboptimality of norm-induced mirror maps for sparse losses}\label{sec: induced-mirror-maps-are-suboptimal}

We give an example where the standard $L_p$ mirror map on the $L_p$ unit ball incurs asymptotically larger regret than a suitable block-norm mirror map on the same feasible set.
The separation holds for realized regret with the standard step sizes used throughout this paper.
Here, the exponent $p>1$ approaches $1$ as the dimension grows.

For $1<p\le2$, write
\[
    B_p^d:=\{x\in\R^d:\|x\|_p\le1\},
    \qquad
    h_p(x):=\frac{\|x\|_p^2}{2(p-1)},
    \qquad
    a:=1-\frac1p.
\]
The mirror map $h_p$ is $1$-strongly convex with respect to the $L_p$ norm; see, e.g., \cite{kwon2016gains}.
All algorithms below start at $x^{(1)}=0$.

\begin{proposition}[Separation on the matching norm ball]\label{prop: fixed-lp-block-gain}
    Let $1<p\le2$, let $3\le s\le d$ be an integer, and suppose that $T\ge8/(p-1)$.
    There exists a deterministic sequence of linear losses on $B_p^d$, with gradients in $[-1,1]^d$ supported on a fixed set of $s$ coordinates, such that OMD with mirror map $h_p$ and standard step size satisfies
    \[
        \regret_{h_p}(T)
        =\Theta\left(s^{1-1/p}\sqrt{\frac{T}{p-1}}\right).
    \]
    On the same instance, OMD with a uniformly random equal-size partition into $s$ blocks, normalized mirror map $h_{s,\rho_s(B_p^d)}$, and standard step size satisfies
    \[
        \E[\regret_s(T)]
        =O\left(s^{1-1/p}\ln s\sqrt T\right).
    \]
    All implicit constants are absolute.
    In particular,
    \[
        \regret_{h_p}(T)
        \ge
        \frac{c}{\sqrt{p-1}\,\ln s}\,
        \E[\regret_s(T)]
    \]
    for an absolute constant $c>0$.
\end{proposition}

\begin{proof}
    Fix a set $C\subseteq[d]$ with $|C|=s$, and use the same loss function at every round:
    \[
        f^{(t)}(x):=-\langle\mathbf1_C,x\rangle,
        \qquad t\in[T].
    \]
    The gradients are bounded and exactly $s$-sparse.
    By H\"older's inequality, the optimal comparator is
    \[
        u:=s^{-1/p}\mathbf1_C,
        \qquad
        \|u\|_p=1,
        \qquad
        f^{(t)}(u)=-s^a.
    \]

    \paragraph{Regret of the $L_p$ mirror map.}
    Since $h_p(0)=0$ and $\nabla h_p(0)=0$, its Bregman diameter from the origin is
    \[
        D_{h_p}^2
        =\max_{x\in B_p^d}B_{h_p}(x\|0)
        =\max_{x\in B_p^d}h_p(x)
        =\frac1{2(p-1)}.
    \]
    Let $q=p/(p-1)$ be the dual exponent.
    The gradient bound and standard step size are therefore
    \[
        G_{h_p}=\|\mathbf1_C\|_q=s^a,
        \qquad
        \eta_p
        =\frac{D_{h_p}}{G_{h_p}\sqrt T}
        =\frac1{s^a\sqrt{2(p-1)T}}.
    \]

    We claim that the iterates have the form $x^{(t)}=r_tu$, where
    \[
        r_1=0,
        \qquad
        r_{t+1}=\min\{1,r_t+\delta\},
        \qquad
        \delta:=\sqrt{\frac{p-1}{2T}}.
    \]
    To see this, suppose that $x^{(t)}=r_tu$.
    The OMD objective is strictly convex and invariant under permutations of coordinates in $C$ and sign changes of coordinates outside $C$.
    Thus, its unique minimizer has equal coordinates on $C$ and zero coordinates outside $C$.
    It consequently lies on the line $\{ru:-1\le r\le1\}$.
    On this line,
    \[
        h_p(ru)=\frac{r^2}{2(p-1)},
        \qquad
        B_{h_p}(ru\|r_tu)=\frac{(r-r_t)^2}{2(p-1)}.
    \]
    Hence the update reduces to
    \[
        r_{t+1}
        =\arg\min_{-1\le r\le1}
        \left\{-s^a r+\frac{(r-r_t)^2}{2\eta_p(p-1)}\right\}
        =\min\{1,r_t+\eta_p(p-1)s^a\}.
    \]
    Since $\eta_p(p-1)s^a=\delta$, induction proves the claim.

    It follows that the realized regret is exactly
    \[
        \regret_{h_p}(T)
        =s^a\sum_{t=1}^T
        \left(1-\min\{1,(t-1)\delta\}\right).
    \]
    Our assumptions imply $\delta\le1/4$ and $1/\delta\le T/2$.
    For the first $m:=\lfloor1/(2\delta)\rfloor$ rounds, each summand is at least $1/2$, and $m\ge1/(4\delta)$.
    Conversely, at most $\lceil1/\delta\rceil$ summands are nonzero, and each is at most $1$.
    Therefore,
    \[
        \frac{s^a}{8\delta}
        \le\regret_{h_p}(T)
        \le\frac{2s^a}{\delta},
    \]
    proving
    \[
        \regret_{h_p}(T)
        =\Theta\left(s^a\sqrt{\frac{T}{p-1}}\right).
    \]

    \paragraph{Regret of the block-norm mirror map.}
    We first compute the block-norm radius of $B_p^d$.
    For any partition $\mathcal B=\{B_1,\ldots,B_n\}$ into $n$ nonempty blocks, the inequality $p\le2$ and H\"older's inequality give
    \[
        \|x\|_{[\mathcal B]}
        =\sum_{j=1}^n\|x_{B_j}\|_2
        \le\sum_{j=1}^n\|x_{B_j}\|_p
        \le n^{1-1/p}
        \left(\sum_{j=1}^n\|x_{B_j}\|_p^p\right)^{1/p}
        \le n^a
    \]
    for every $x\in B_p^d$.
    Equality is attained by placing mass $n^{-1/p}$ on one coordinate in each block.
    Consequently,
    \[
        \rho_n(B_p^d)=n^a.
    \]

    Choose $n=s$ blocks.
    By \cref{lem: block-norm-diameter-bound-general,lem: sparse-vector-block-variance-bound},
    \[
        D_s=O(s^a\sqrt{\ln s}),
        \qquad
        G_s=O(\sqrt{\ln s}).
    \]
    The standard expected-regret guarantee therefore yields
    \[
        \E[\regret_s(T)]
        \le2D_sG_s\sqrt T
        =O(s^a\ln s\sqrt T).
    \]
    Combining this upper bound with the realized-regret lower bound for $h_p$ proves the result.
\end{proof}

\paragraph{Logarithmic separation.}
Take
\[
    p=1+\frac1{\ln d},
    \qquad
    s=\lfloor\ln d\rfloor,
    \qquad
    T\ge8\ln d.
\]
Then $a=1/(\ln d+1)$ and $s^a=\Theta(1)$.
For sufficiently large $d$, the proposition gives
\[
    \regret_{h_p}(T)=\Theta(\sqrt{T\ln d}),
    \qquad
    \E[\regret_s(T)]=O(\sqrt T\ln\ln d).
\]
Thus, the block-norm mirror map improves realized regret by a factor of at least $\Omega(\sqrt{\ln d}/\ln\ln d)$.

\paragraph{Polynomial separation.}
Alternatively, take
\[
    p=1+d^{-1/2},
    \qquad
    s=\lfloor\sqrt d\rfloor,
    \qquad
    T\ge8\sqrt d.
\]
Now $a=1/(\sqrt d+1)$, so again $s^a=\Theta(1)$.
The proposition gives
\[
    \regret_{h_p}(T)=\Theta(d^{1/4}\sqrt T),
    \qquad
    \E[\regret_s(T)]=O(\sqrt T\ln d).
\]
Hence, even when the feasible set is precisely the unit ball of the norm defining $h_p$, a suitable block-norm mirror map improves realized regret by a factor of at least $\Omega(d^{1/4}/\ln d)$.

\section{Gain computations for the remaining convex sets}\label{app: block-norm-computations}

We follow the template of the main text; each paragraph proves the general-parameter, general-$s$ row of \cref{tab: gains} for one family, and ends by specializing the parameters and $s$ to recover the second row.
We use the normalized mirror maps $h_{n,\rho_n(\convexbody)}$ and the uniform gradient factors $G_n(s)$, abbreviated below by $G_n$, for gradients in $[-1,1]^d$.
Each body in the examples below is symmetric about the origin, and the diameter bounds hold for any fixed starting point $x^{(1)}\in\convexbody$.

Recall that the $1$st block norm is the $L_2$ norm and the $d$th block norm (singleton blocks) is the $L_1$ norm.
We compute the endpoint lower bounds directly.
For either endpoint mirror map $h$ and corresponding norm $\|\cdot\|$, put $a=x^{(1)}$.
Given $z\in\convexbody$, symmetry also gives $-z\in\convexbody$.
Strong convexity and the triangle inequality imply
\begin{align*}
    2D_h^2&\ge B_h(z\|a)+B_h(-z\|a)\\
    &\ge\frac12\bigl(\|z-a\|^2+\|-z-a\|^2\bigr)\\
    &\ge\frac14\bigl(\|z-a\|+\|-z-a\|\bigr)^2
    \ge\|z\|^2.
\end{align*}
Thus exhibiting a point of $L_2$ norm at least $R_\euc$ and a point of $L_1$ norm at least $R_\ent$ gives $D_\euc\ge R_\euc/\sqrt2\ge R_\euc/2$ and $D_\ent\ge R_\ent/\sqrt2\ge R_\ent/2$.
If, in addition, $M(n)\ge \max_{x\in\convexbody}\|x\|_{[n]}$ holds for \emph{every} partition of $[d]$ into at most $n$ nonempty blocks, then it bounds the radius $\rho_n(\convexbody)\le M(n)$, since this includes every equal-size $n$-partition.
Hence \cref{lem: block-norm-diameter-bound-general} gives $D_n=O(\rho_n(\convexbody)\sqrt{\ln n})=O(M(n)\sqrt{\ln n})$.
Choosing $n=s$ and using $G_\euc= \sqrt s$, $G_\ent=1$ and $G_s(s)\le3\sqrt{\ln s}$ (for $s\ge3$, by \cref{lem: sparse-vector-block-variance-bound}), we obtain the master bound
\begin{equation}\label{eqn: master-gain}
    \gain(\convexbody,s)
    \;\ge\;\frac{\min\{D_\euc G_\euc,\ D_\ent G_\ent\}}{D_sG_s}
    \;=\;\Omega\left(\frac{\min\{R_\euc\sqrt s,\ R_\ent\}}
    {M(s)\,\ln s}\right).
\end{equation}
Each paragraph therefore only has to (i) record the normalization making $\convexbody\subseteq\{x:\|x\|_1\le1\}$ -- this is without loss of generality when the starting point is scaled with the body.
Indeed, $\rho_n(c\convexbody)=c\rho_n(\convexbody)$ and the normalized maps satisfy
\[
    B_{h_{\mathcal B,\rho_n(c\convexbody)}}(cz\|cx^{(1)})
    =c^2B_{h_{\mathcal B,\rho_n(\convexbody)}}(z\|x^{(1)}).
\]
Thus every $D_n$ scales by $c$, while the uniform factors $G_n(s)$ are independent of the body, so the common factor cancels in the gain.
The proof of \cref{eqn: master-gain} still uses the individual radius $\rho_n(\convexbody)$ for each mirror map -- then (ii) exhibit the witnesses for $R_\euc$ and $R_\ent$, and (iii) bound $M(n)$.

For the bodies built as Minkowski sums we use the triangle inequality $\max_{\convexbody}\|\cdot\|_{[n]}\le\sum_m\alpha_m \max_{A_m}\|\cdot\|_{[n]}$ over $\convexbody=\sum_m\alpha_mA_m$, together with the following block norms of the pieces, valid for every partition into $N\le n$ nonempty blocks $B_1,\dots,B_N$ of sizes $b_1,\dots,b_N$:
\begin{equation}\label{eqn: piece-norms}
    \max_{cB_1}\|\cdot\|_{[n]}
    =\max_{c[-\mathbf{e}_1,\mathbf{e}_1]}\|\cdot\|_{[n]}=c,
    \qquad
    \max_{cB_2}\|\cdot\|_{[n]}\le c\sqrt n,
    \qquad
    \max_{cB_\infty}\|\cdot\|_{[n]}
    =c\sum_i\sqrt{b_i}\le c\sqrt{nd}.
\end{equation}
Indeed, the block norm is convex and nondecreasing in coordinate magnitudes, so each maximum is attained at a vertex (for $B_1$, the $1$-sparse points $\pm c\mathbf{e}_j$, of block norm $c$) or at the corner $c\mathbf1$ (for $B_\infty$, giving $c\sum_i\sqrt{b_i}\le c\sqrt{n\sum_ib_i}$ by Cauchy--Schwarz); for $B_2$, writing $\tau_i:=\|x_{B_i}\|_2$, the constraint is $\sum_i\tau_i^2\le c^2$ and hence $\sum_i\tau_i\le c\sqrt n$.

\paragraph{$L_p$ norm ball, $1\le p\le2$.}
Normalize
$\convexbody=c\,B_ p$ with $c:=d^{-(1-1/p)}$.
We show
\[
    \gain(cB_ p,s)
    =\Omega\left(\frac{\min\bigl(\sqrt s,\ d^{1-1/p}\bigr)}
    {s^{1-1/p}\,\ln s}\right).
\]
\emph{Containment:} by H\"older, $\|x\|_1\le d^{1-1/p}\|x\|_p\le cd^{1-1/p}=1$.
\emph{Block norm:} for every partition with $N\le n$ nonempty blocks, within each block $\|x_{B_i}\|_2\le\|x_{B_i}\|_p=:c_i$ (as $p\le2$), and H\"older across blocks gives $\sum_ic_i\le N^{1-1/p}\bigl(\sum_ic_i^p\bigr)^{1/p}\le cN^{1-1/p}$; this is attained by placing mass $cN^{-1/p}$ on one coordinate per nonempty block.
Hence $M(n)=cn^{1-1/p}$.
Since the same construction works for every equal-size $n$-partition, $\rho_n(\convexbody)=cn^{1-1/p}$, and the normalized diameter bound gives $D_n=O(cn^{1-1/p}\sqrt{\ln n})$.
\emph{Witnesses ($0\in\convexbody$):} $c\mathbf{e}_1$ gives $R_\euc=c$ (this is exact, since $\|x\|_2\le\|x\|_p\le c$); the flat point $cd^{-1/p}\mathbf1$ has $\|\cdot\|_p=c$ and gives $R_\ent=cd^{1-1/p}=1$.
\emph{Gain:} plug into \cref{eqn: master-gain}; the factor $c$ cancels.
\emph{Specialization ($p=4/3$, $s=\sqrt d$):} $\min(\sqrt s,d^{1/4})=d^{1/4}$ and $s^{1/4}=d^{1/8}$, so the gain is $\Omega\bigl(d^{1/8}/\ln d\bigr)$.

\paragraph{Box.}
Let $\convexbody=\{x:|x_i|\le w_i\ \forall i\}$ with
$w_1\ge\dots\ge w_d\ge0$ and $w\ne0$, normalized so that $\|w\|_1\le1$.
We show
\[
    \gain(\convexbody,s)
    =\Omega\left(\frac{\min\bigl(\|w\|_2\sqrt s,\ \|w\|_1\bigr)}
    {\bigl(\|w_{\le s}\|_1+\sqrt s\,\|w_{>s}\|_2\bigr)\ln s}\right).
\]
\emph{Containment:} $\|x\|_1\le\|w\|_1\le1$.
\emph{Block norm:} the block norm is nondecreasing in coordinate magnitudes, so its maximum over the box is at the corner $w$; splitting $w$ at rank $n$ and using the triangle inequality (head in $L_1$) plus Cauchy--Schwarz over the at most $n$ nonempty blocks (tail in $L_2$),
\[
    \|w\|_{[n]}\;\le\;\|w_{\le n}\|_1+\sqrt n\,\|w_{>n}\|_2
    \;=:\;M(n).
\]
This bound holds for every partition, so $\rho_n(\convexbody)\le M(n)$ and $D_n=O(M(n)\sqrt{\ln n})$.
\emph{Witnesses ($0\in\convexbody$):} the corner $w$ serves for both: $R_\euc=\|w\|_2$ and $R_\ent=\|w\|_1$.
\emph{Gain:} plug into \cref{eqn: master-gain}.
\emph{Specialization ($w_1=\tfrac12d^{-1/4}$, $w_j=\tfrac1{2d}$ for $j\ge2$, $s=\sqrt d$):} $\|w\|_2\ge w_1=\tfrac12d^{-1/4}$ and $\|w\|_1\ge\tfrac{d-1}{2d}\ge\tfrac14$, so the numerator is at least $\tfrac14$; and $M(s)\le\bigl(\tfrac12d^{-1/4}+\tfrac{\sqrt d}{2d}\bigr) +d^{1/4}\cdot\tfrac{\sqrt d}{2d}\le\tfrac32d^{-1/4}$.
Hence the gain is $\Omega\bigl(d^{1/4}/\ln d\bigr)$.

\paragraph{Sum of $L_1$ and $L_\infty$ balls.}
Let
$\convexbody=\alpha_1 B_1+\alpha_\infty B_\infty$ with $\alpha_1,\alpha_\infty>0$, normalized so that $\alpha_1+\alpha_\infty d\le1$.
We show
\[
    \gain(\convexbody,s)
    =\Omega\left(\frac{\min\bigl((\alpha_1+\alpha_\infty\sqrt d)
        \sqrt s,\ \alpha_1+\alpha_\infty d\bigr)}
    {\bigl(\alpha_1+\alpha_\infty\sqrt{sd}\bigr)\ln s}\right).
\]
\emph{Containment:} for $x=u+u'$ with $\|u\|_1\le\alpha_1$, $\|u'\|_\infty\le\alpha_\infty$: $\|x\|_1\le\alpha_1+\alpha_\infty d\le1$.
\emph{Block norm:} by \cref{eqn: piece-norms} and the triangle inequality, $M(n)=\alpha_1+\alpha_\infty\sqrt{nd}$.
Thus $\rho_n(\convexbody)\le M(n)$ and $D_n=O((\alpha_1+\alpha_\infty\sqrt{nd})\sqrt{\ln n})$.
\emph{Witnesses ($0\in\convexbody$):} the single point $z:=\alpha_1\mathbf{e}_1+\alpha_\infty\mathbf1\in\convexbody$ serves for both: $\|z\|_2\ge\sqrt{\alpha_1^2+(d-1)\alpha_\infty^2} \ge\tfrac12\bigl(\alpha_1+\alpha_\infty\sqrt d\bigr)=:R_\euc$ (for $d\ge2$, using $\sqrt{a^2+b^2}\ge\tfrac{a+b}{\sqrt2}$ and $\sqrt{d-1}\ge\sqrt{d/2}$), and $\|z\|_1=\alpha_1+\alpha_\infty d=:R_\ent$.
\emph{Gain:} plug into \cref{eqn: master-gain}.
\emph{Specialization ($\alpha_1=\tfrac12d^{-1/4}$, $\alpha_\infty=\tfrac1{2d}$, $s=\sqrt d$):} $R_\euc\sqrt s\ge\tfrac12\alpha_1\sqrt s=\tfrac14$ and $R_\ent\ge\tfrac12$, while $M(s)=\tfrac12d^{-1/4}+\tfrac{d^{3/4}}{2d}=d^{-1/4}$; the gain is $\Omega\bigl(d^{1/4}/\ln d\bigr)$.

\paragraph{Sum of $L_1$ and $L_2$ balls.}
Let
$\convexbody=\alpha_1B_1+\alpha_2B_2$ with $\alpha_1,\alpha_2>0$, normalized so that $\alpha_1+\alpha_2\sqrt d\le1$.
We show
\[
    \gain(\convexbody,s)
    =\Omega\left(\frac{\min\bigl((\alpha_1+\alpha_2)\sqrt s,\
        \alpha_1+\alpha_2\sqrt d\bigr)}
    {\bigl(\alpha_1+\alpha_2\sqrt s\bigr)\ln s}\right).
\]
\emph{Containment:} for $x=u+u'$ with $\|u\|_1\le\alpha_1$, $\|u'\|_2\le\alpha_2$: $\|x\|_1\le\alpha_1+\alpha_2\sqrt d\le1$.
\emph{Block norm:} by \cref{eqn: piece-norms}, $M(n)=\alpha_1+\alpha_2\sqrt n$.
Consequently $\rho_n(\convexbody)\le M(n)$ and $D_n=O((\alpha_1+\alpha_2\sqrt n)\sqrt{\ln n})$.
\emph{Witnesses ($0\in\convexbody$):} the aligned point $(\alpha_1+\alpha_2)\mathbf{e}_1\in\convexbody$ gives $R_\euc=\alpha_1+\alpha_2$ (this is exact, by the triangle inequality); the point $\alpha_1\mathbf{e}_1+\tfrac{\alpha_2}{\sqrt d}\mathbf1$ gives $R_\ent=\alpha_1+\alpha_2\sqrt d$.
\emph{Gain:} plug into \cref{eqn: master-gain}.
\emph{Specialization ($\alpha_1=\tfrac12d^{-1/4}$, $\alpha_2=\tfrac12d^{-1/2}$, $s=\sqrt d$):} $R_\euc\sqrt s\ge\tfrac12$ and $R_\ent\ge\tfrac12$, while $M(s)=\tfrac12d^{-1/4}+\tfrac12d^{-1/2}\cdot d^{1/4}=d^{-1/4}$; the gain is $\Omega\bigl(d^{1/4}/\ln d\bigr)$.

\paragraph{Capsule (segment plus $L_2$ ball).}
Let
$\convexbody=\alpha_0[-\mathbf{e}_1,\mathbf{e}_1]+\alpha_2B_2$ with $\alpha_0,\alpha_2>0$, normalized so that $\alpha_0+\alpha_2\sqrt d\le1$.
A segment contributes to every quantity above exactly as an $L_1$ ball of the same radius: its block norm is $\alpha_0$ for every partition \cref{eqn: piece-norms}, and the witnesses $(\alpha_0+\alpha_2)\mathbf{e}_1$ and $\alpha_0\mathbf{e}_1+\tfrac{\alpha_2}{\sqrt d}\mathbf1$ lie in $\convexbody$.
In particular, $\rho_n(\convexbody)\le\alpha_0+\alpha_2\sqrt n$, and the same endpoint witnesses and their negatives give the required lower bounds.
Hence the previous paragraph applies verbatim with $\alpha_1$ replaced by $\alpha_0$:
\[
    \gain(\convexbody,s)
    =\Omega\left(\frac{\min\bigl((\alpha_0+\alpha_2)\sqrt s,\
        \alpha_0+\alpha_2\sqrt d\bigr)}
    {\bigl(\alpha_0+\alpha_2\sqrt s\bigr)\ln s}\right),
\]
which at $\alpha_0=\tfrac12d^{-1/4}$, $\alpha_2=\tfrac12d^{-1/2}$, $s=\sqrt d$ gives gain $\Omega\bigl(d^{1/4}/\ln d\bigr)$.

\section{Omitted proofs from \cref{sec: instances-with-realized-gains}}\label{sec: polynomial-regret-improvement-proofs}

In this section, we supply proofs omitted from \cref{sec: instances-with-realized-gains}.

First, we prove the regret upper bound $\E[\regret_s(T)] = O(\sqrt T\ln d)$ for the \gls{oco} instance in \cref{sec: polynomial-regret-improvement}.

Recall that $P=\conv(\simplex_d\cup\{A\mathbf1_d\})$, where $A=d^{-3/4}$ and $s=\lfloor\sqrt d\rfloor$.
Since any norm is convex, its maximum over $P$ occurs at a vertex.
For every partition into $s$ blocks,
\[
    \|\mathbf e_i\|_{[s]}=1,\qquad
    \|A\mathbf1_d\|_{[s]}=A\sum_{B\in\mathcal B}\sqrt{|B|}
    \le A\sqrt{sd}\le1.
\]
Thus $\rho_s(P)=1$, so the normalized mirror map is $h_s$ itself.
By \cref{lem: sparse-vector-block-variance-bound} and \cref{lem: block-norm-diameter-bound-general},
\[
    G_s=O(\sqrt{\ln s}),\qquad D_s=O(\sqrt{\ln s}).
\]
The regret is therefore
\[
    \E[\regret_s(T)]\le 2D_sG_s\sqrt T
    =O(\sqrt T\ln s)=O(\sqrt T\ln d).
\]

\subsection{Euclidean projection coordinate bound}
\label{app: polynomial-regret-euclidean-coordinate-bound}

Recall that $P=\conv(\simplex_d\cup\{A\mathbf1_d\})$, where
$A=d^{-3/4}$ and $d>1$, so $Ad>1$.
We show that, for $x\in P$, $c\in\{0,1\}^d$, and $\eta\ge0$,
the Euclidean projection $x'=\Pi_P(y)$ of $y=x+\eta c$ satisfies
\[
    x_i'\le\max\{A,y_i\}\qquad\text{for every }i\in[d].
\]

\begin{proof}
    Write $x'=A\alpha\mathbf1_d+\lambda$, where $\alpha,\lambda_i\ge0$ and $\alpha+\sum_i\lambda_i=1$.
    The projection minimizes $\tfrac12\|A\alpha\mathbf1_d+\lambda-y\|_2^2$ over these coefficients.
    Put $q_i:=x_i'-y_i$.
    The optimality conditions give a multiplier $\mu$ for the equality constraint such that
    \[
        q_i+\mu\ge0,\qquad \lambda_i(q_i+\mu)=0, \qquad A\sum_iq_i+\mu\ge0,\qquad \alpha\left(A\sum_iq_i+\mu\right)=0.
    \]
    We claim that $\mu\ge0$.
    Otherwise, $q_i\ge-\mu>0$ for every $i$, and
    \[
        A\sum_iq_i+\mu\ge(-\mu)(Ad-1)>0.
    \]
    Complementary slackness then forces $\alpha=0$, so $\sum_i x_i'=1$.
    But $x\in P$ implies $\sum_i x_i\ge1$, and $c\ge0$ implies $\sum_i y_i\ge1$.
    This contradicts $x_i'>y_i$ for every $i$, proving the claim.

    If $\lambda_i>0$, then $x_i'=y_i-\mu\le y_i$.
    If $\lambda_i=0$, then $x_i'=A\alpha\le A$.
    These two cases prove the desired bound.
\end{proof}

In particular, $c_1^{(t)}=1$ gives $x_1^{(t+1)}\le\max\{A,x_1^{(t)}+\eta\}$.
Starting from $x_1^{(1)}=A$, induction yields $x_1^{(t)}\le A+(t-1)\eta$, as used in \cref{sec: polynomial-regret-improvement}.

\subsection{Proof of \cref{lem: polynomial-lower-bound-entropic-coordinate-bound}}\label{app: mirror-maps-proof-of-polynomial-regret-coordinatewise-bound}

\begin{proof}
    Recall that $R=Ad=d^{1/4}$, $\hat P=P/R$, and $\var{z}{1}=\mathbf1_d/d$.
    First, we compute the step size.
    Since $B_{h_d}(z\|\var{z}{1})$ is convex in $z$, its maximum over $\hat P$ occurs at a vertex.
    By symmetry, using $p_d=1+1/\ln d$,
    \begin{align*}
        \hat D_d^2
        &=B_{h_d}\left(\frac{\mathbf e_1}{R}\middle\|\frac{\mathbf1_d}{d}\right)\\
        &=\frac{e\ln d}{p_d}R^{-p_d}+\frac1{p_d}-\frac{\ln d}{R}\\
        &=\frac1{p_d}+\frac{\ln d}{R}
        \left(\frac{eR^{-1/\ln d}}{p_d}-1\right)=\Theta(1).
    \end{align*}
    Here $R^{-1/\ln d}=e^{-1/4}$ and $(\ln d)/R=o(1)$.
    Since $D_\ent=R\hat D_d$ and $G_\ent=1$, the original step size satisfies $\eta=R\hat D_d/\sqrt T=\Theta(R/\sqrt T)$.
    The scaled step size is $\hat\eta=\eta/R^2$; denote $\eta_0:=R\hat\eta=\eta/R=\Theta(T^{-1/2})$.

    We will prove that
    \begin{equation}\label{eqn: polynomial-regret-scaled-iterate-bound}
        \var{z}{t}_1\le\frac1d+\frac{\eta_0}{R^2}(t-1).
    \end{equation}
    Every point $z\in\hat P$ has a unique representation
    \[
        z=b\mathbf1_d+\frac{\lambda}{R},\qquad
        b,\lambda_i\ge0,\qquad db+\sum_i\lambda_i=1.
    \]
    In particular, $0\le b\le1/d$ and
    \[
        \sum_i z_i=\frac1R+d\left(1-\frac1R\right)b.
    \]
    Write $\psi(u):=e\ln d\,u^{1/\ln d}$, so that $(\nabla h_d(z))_i=\psi(z_i)$.
    Consider one OMD step
    \[
        z'=\arg\min_{w\in\hat P}
        \big\{B_{h_d}(w\|z)-\eta_0\langle c,w\rangle\big\},
        \qquad c\in\{0,1\}^d,
    \]
    and denote its coefficients by $b',\lambda'$.
    Put $q_i:=\psi(z_i')-\psi(z_i)-\eta_0c_i$.
    The optimality conditions in $b',\lambda'$ give a multiplier $\mu$ satisfying
    \begin{equation}\label{eqn: polynomial-regret-projection-optimality}
        q_i=-\mu\ \text{if }\lambda_i'>0,\qquad
        q_i\ge-\mu\ \text{if }\lambda_i'=0,\qquad
        \sum_iq_i\ge-\frac{d\mu}{R},
    \end{equation}
    where the last inequality is an equality if $b'>0$.
    Indeed, writing the multiplier of $db'+\sum_i\lambda_i'=1$ as $\mu/R$, stationarity gives $q_i/R+\mu/R\ge0$ and $\sum_iq_i+d\mu/R\ge0$, with equality for positive $\lambda_i'$ and $b'$, respectively.

    Suppose $b'<b$, and denote $\delta:=\psi(b)-\psi(b')>0$.
    Some $\lambda_i'$ must be zero: otherwise $q_i=-\mu$ for all $i$, and \cref{eqn: polynomial-regret-projection-optimality} implies $\mu\le0$ since $R>1$.
    This would give $\psi(z_i')-\psi(z_i)=\eta_0c_i-\mu\ge0$ for every $i$, contradicting $\sum_i z_i'<\sum_i z_i$.

    For every coordinate with $\lambda_i'=0$, we have $z_i'=b'$ and $z_i\ge b$, so $q_i\le-\delta$.
    Since also $q_i\ge-\mu$, this implies $\mu\ge\delta$.
    Hence $q_i\le-\delta$ on every coordinate, including those with $\lambda_i'>0$.
    Summing and using \cref{eqn: polynomial-regret-projection-optimality},
    \[
        -\frac{d\mu}{R}\le\sum_iq_i\le-d\delta,
        \qquad\text{so}\qquad\delta\le\frac\mu R.
    \]
    Further, $\mu>\eta_0$ would imply $z_i'<z_i$ on all coordinates: on free coordinates this follows from $\psi(z_i')-\psi(z_i)=\eta_0c_i-\mu<0$, and on the others from $z_i'=b'<b\le z_i$.
    Since $c\ge0$, this cannot improve the linear loss and has strictly positive Bregman divergence, contradicting optimality relative to staying at $z$.
    Thus $\mu\le\eta_0$, proving
    \[
        \psi(b)-\psi(b')\le\frac{\eta_0}{R}.
    \]
    This inequality also holds when $b'\ge b$.
    Therefore, writing $b^{(t)}$ for the coefficient of $\var{z}{t}$ and using $b^{(1)}=1/d$, telescoping gives
    \[
        \psi(1/d)-\psi(b^{(t)})\le\frac{(t-1)\eta_0}{R}.
    \]
    Since $\psi'(u)=e u^{1/\ln d-1}\ge d$ for $0<u\le1/d$ and $d\ge3$, it follows that
    \[
        1-db^{(t)}\le\frac{(t-1)\eta_0}{R}.
    \]
    Finally,
    \[
        R\var{z}{t}_1=Rb^{(t)}+\lambda_1^{(t)}
        \le A+(1-db^{(t)})
        \le A+\frac{(t-1)\eta_0}{R}.
    \]
    This proves \cref{eqn: polynomial-regret-scaled-iterate-bound} and, since $\var{x}{t}=R\var{z}{t}$ and $\eta_0=\eta/R$, the claimed bound.
\end{proof}

\section{Algorithm \texttt{MirrorWeights} for learning mirror map online}\label{app: mwu-proof}

Here, we present \cref{alg: mirror-weights}, which, given an OCO instance and a set of mirror maps, learns the optimal mirror map from this set online.
We then prove \cref{thm: mirror-weights} that formalizes this guarantee for the algorithm.
The proof uses standard arguments for the Hedge algorithm \cite{hazan_introduction_2016}; we present it here for completeness.
We restate the result here:
\combiningMirrorMaps*

\begin{algorithm}[t]
    \caption{\texttt{MirrorWeights}}\label{alg: mirror-weights}
    \KwIn{convex body $\mathcal{K} \subseteq \R^d$, time horizon $T$, starting point $x^{(1)} \in \mathcal{K}$, strongly convex distance generating functions $m_1, \ldots, m_N$, upper bound $\rho \ge \max_{x, z \in \mathcal{K}, t \in [T]} f^{(t)}(x) - f^{(t)}(z)$ on the differential in loss function values}
    \let\oldnl\nl
    \newcommand{\nonl}{\renewcommand{\nl}{\let\nl\oldnl}}
    \nonl \textbf{Parameters}: step sizes $\{\eta^{(t)}_\ell, \ell \in [N], t \in [T]\}$, and learning rate $\varepsilon > 0$
    initialize $X \in \R^{d \times N}$ with all columns $X_\ell \gets x^{(1)}$\;
    initialize probabilities $p_\ell \gets \tfrac{1}{N}$ for all $\ell \in [N]$\;
    \For{$t = 1, \ldots, T$}{
        play $x^{(t)} = \sum_{\ell=1}^N p_\ell X_\ell$ and observe loss $f^{(t)}$\;
        \For{$\ell = 1, \ldots, N$}{
            $p_\ell \gets p_\ell \cdot \exp\left(- \frac{\varepsilon f^{(t)}(X_\ell)}{\rho}\right)$\;
            $X_\ell \gets \texttt{MirrorDescentStep}(X_\ell, \mathcal{K}, m_\ell, \nabla f^{(t)}, \eta^{(t)}_\ell)$\;
        }
        normalize $p \gets p / \|p\|_1$\;
    }
\end{algorithm}

\begin{proof}
    Let $\ell^* = {\arg\min}_{\ell \in [N]} \sum_{t \in [T]} f^{(t)}(X_\ell^{(t)})$, and denote $X_{\ell^*}^{(t)} = X_*^{(t)}$ for all $t$.
    Define \emph{weights} $w_{\ell}^{(t)}$ inductively as follows: $w_{\ell}^{(1)} = \frac{1}{N}$ for all $\ell \in [N]$, and $w_{\ell}^{(t + 1)} = w_\ell^{(t)} \exp\left( - \frac{\varepsilon (f^{(t)}(X_\ell^{(t)}) - f^{(t)}(X^{(t)}_*))}{\rho}\right)$ for all $t$, where $X_\ell^{(t)}$ is the iterate of the $\ell$th mirror map at time $t$.

    Define potential at time $t$ as $\phi^{(t)} = \sum_{\ell \in [N]} w_\ell^{(t)} = \|w^{(t)}\|_1$.
    Note that $p_{\ell}^{(t)} \propto w_\ell^{(t)}$ and therefore $p_{\ell}^{(t)} = \frac{w_\ell^{(t)}}{\|w^{(t)}\|_1} = \frac{w_\ell^{(t)}}{\phi^{(t)}}$.
    Then,
    \begin{align*}
        \phi^{(t + 1)} &= \sum_{\ell \in [N]} w_\ell^{(t + 1)} = \sum_{\ell \in [N]} w_\ell^{(t)} \exp\left( - \frac{\varepsilon (f^{(t)}(X_\ell^{(t)}) - f^{(t)}(X^{(t)}_*))}{\rho}\right).
    \end{align*}
    Further, $|f^{(t)}(X_\ell^{(t)}) - f^{(t)}(X^{(t)}_*)| \le \rho$ by definition and $\varepsilon \in [0, 1]$.
    Since $\exp(u) \le 1 + u + u^2$ for all $u \in [-1, 1]$, we get
    \begin{align*}
        \phi^{(t + 1)} &\le \sum_{\ell \in [N]} w_\ell^{(t)} \left(1 - \frac{\varepsilon (f^{(t)}(X_\ell^{(t)}) - f^{(t)}(X^{(t)}_*))}{\rho} + \frac{\varepsilon^2(f^{(t)}(X_\ell^{(t)}) - f^{(t)}(X^{(t)}_*))^2}{\rho^2}\right) \\
        &\le \sum_{\ell \in [N]} w_\ell^{(t)} \left(1 - \frac{\varepsilon (f^{(t)}(X_\ell^{(t)}) - f^{(t)}(X^{(t)}_*))}{\rho} + \varepsilon^2\right) \\
        &= (1 + \varepsilon^2) \phi^{(t)} - \frac{\varepsilon}{\rho} \sum_{\ell} w_\ell^{(t)} (f^{(t)}(X_\ell^{(t)}) - f^{(t)}(X^{(t)}_*)) \\
        &= (1 + \varepsilon^2) \phi^{(t)} - \frac{\varepsilon \phi^{(t)}}{\rho} \sum_{\ell} p_\ell^{(t)} (f^{(t)}(X_\ell^{(t)}) - f^{(t)}(X^{(t)}_*)).
    \end{align*}
    Recall that the algorithm plays $x^{(t)} = \sum_{\ell} p_\ell^{(t)} X_\ell^{(t)}$.
    Since $f^{(t)}$ is convex, $f^{(t)}(x^{(t)}) \le \sum_{\ell} p_\ell^{(t)} f^{(t)}(X_\ell^{(t)})$.
    Therefore,
    \begin{align*}
        \phi^{(t + 1)} &\le \phi^{(t)} \left((1 + \varepsilon^2) - \frac{\varepsilon(f^{(t)}(x^{(t)}) - f^{(t)}(X_*^{(t)}))}{\rho}\right).
    \end{align*}
    Using $1 + u \le \exp(u)$ for all $u \in \R$,
    \begin{align*}
        \phi^{(t + 1)} &\le \phi^{(t)} \exp\left(\varepsilon^2 - \frac{\varepsilon(f^{(t)}(x^{(t)}) - f^{(t)}(X_*^{(t)}))}{\rho}\right).
    \end{align*}
    Therefore,
    \begin{align*}
        \frac{\phi^{(T + 1)}}{\phi^{(1)}} &\le \exp\left(T\varepsilon^2 - \frac{\varepsilon \sum_{t \in [T]} (f^{(t)}(x^{(t)}) - f^{(t)}(X_*^{(t)}))}{\rho}\right) \\
        &= \exp\left(T\varepsilon^2 - \frac{\varepsilon(\regret(T) - \regret_*(T))}{\rho}\right).
    \end{align*}
    Note that $\phi^{(1)} = 1$ and $\phi^{(T + 1)} \ge w_{*}^{(T + 1)} = w_*^{(1)} = \frac{1}{N}$, so that $- \ln N \le T \varepsilon^2 - \frac{\varepsilon(\regret(T) - \regret_*(T))}{\rho}$.
    Rearranging,
    \[
        \regret(T) \le \regret_*(T) + \rho\left(\frac{\ln N}{\varepsilon} + T \varepsilon\right).
    \]
    By assumption, $T \ge \ln N$.
    Choose $\varepsilon = \sqrt{\frac{\ln N}{T}} \le 1$ to get
    \[
        \regret(T) \le \regret_*(T) + 2 \rho \sqrt{\ln N} \sqrt{T}.
    \]
\end{proof}

\section{Numerical experiment}\label{sec: experiment-details}

We construct a small numerical example to further illustrate improvements in regret using block norms.
See \cref{fig: simplex-regret-vs-block-norms-rate-agnostic} for the regret attained by \gls{omd} with different $n$th block norms for $n \in \{2^0, 2^1, \ldots, 2^{12}\}$, over the 4096-dimensional probability simplex, with the uniform starting point $(1/d) \mathbf{1}_d$.
The time horizon $T = 250$, and for each $t \in [T]$, the loss function $f^{(t)}(x) = - \langle c^{(t)}, x \rangle$ where $c^{(t)} \in \{0, 1\}^d$ is an $s$-sparse vector where $s = \lfloor \ln d \rfloor$.
The choice of vectors $c^{(t)}$ is described next:

At each time step $t \in [T]$, we choose a special coordinate $i(t) \in [d]$ such that $c^{(t)}_{i(t)} = 1$ and the remaining $s - 1$ non-zero coordinates of $c^{(t)}$ are chosen uniformly at random from among the remaining $d - 1$ coordinates.
$i(t)$ is chosen as follows: fix $T_0 = \lfloor 2 \sqrt{T} \rfloor$.
If $t \le T_0$, then $i(t)$ alternates between $1$ and $2$, and for $t > T_0$, $i(t)$ alternates between $3$ and $4$.
That is,
\begin{equation*}
    i(t) =
    \begin{cases}
        1 & \text{if} \ t \le T_0 \ \text{and} \ t \ \text{is odd}, \\
        2 & \text{if} \ t \le T_0 \ \text{and} \ t \ \text{is even}, \\
        3 & \text{if} \ t > T_0 \ \text{and} \ t \ \text{is odd}, \\
        4 & \text{if} \ t > T_0 \ \text{and} \ t \ \text{is even}.
    \end{cases}
\end{equation*}

As \cref{fig: simplex-regret-vs-block-norms-rate-agnostic} shows, using $n = 16$ blocks leads to a regret of $< 12.5$, as opposed to a regret of $> 21$ for either \gls{opgd} or \gls{oeg}.
This is in line with what is expected from our theoretical results and shows the usefulness of block norms in certain sparsity regimes of the loss functions, averaged across 10 independent runs of the experiment.

The code for this experiment is available at \url{https://github.com/jaimoondra/block-norm-omd}.

\end{document}